\documentclass[12pt,reqno]{amsart}
\usepackage{amssymb,amsfonts,cite}
\usepackage{amsthm,amsmath}
\usepackage{hyperref}
\usepackage{enumitem}
\usepackage[left=1 in,top=1 in,right=1 in,bottom=1 in]{geometry}

\newcommand{\beq}{\begin{equation}}
\newcommand{\eeq}{\end{equation}}
\newcommand{\beqs}{\begin{equation*}}
\newcommand{\eeqs}{\end{equation*}}
\newcommand{\beqstar}{\begin{equation*}}
\newcommand{\eeqstar}{\end{equation*}}
\newcommand{\ba}{\begin{array}}
\newcommand{\ea}{\end{array}}
\newcommand{\beas}{\begin{eqnarray*}}
\newcommand{\eeas}{\end{eqnarray*}}
\newcommand{\bea}{\begin{eqnarray}}
\newcommand{\eea}{\end{eqnarray}}

\newcommand{\al}{\alpha}

\newcommand{\R}{\ensuremath{\mathbb R}}
\newcommand{\K}{\ensuremath{\mathbb K}}
\newcommand{\C}{\ensuremath{\mathbb C}}
\newcommand{\N}{\ensuremath{\mathbb N}}
\newcommand{\Z}{\ensuremath{\mathbb Z}}
\newcommand{\mD}{\mathcal D}

\newcommand{\bvec}[1]{\mathbf{#1}}
\def\vecx{\bvec x}
\def\vecy{\bvec y}
\def\vecz{\bvec z}

\def\vecb{\bvec b}

\def\vecf{\bvec f}
\def\vecg{\bvec g}
\def\vecQ{\bvec Q}

\def\vecU{\bvec U}
\def\vecV{\bvec V}
\def\vece{\bvec e}
\def\vecu{\bvec u}
\def\vecv{\bvec v}
\def\vecx{\bvec x}
\def\vecw{\bvec w}
\def\vecX{\bvec X}

\def\veck{\bvec k}
\def\vecj{\bvec j}

\def\vecL{\bvec L}

\def\vecm{\bvec m}

\newcommand{\inprod}[1]{\langle{#1}\rangle}
\newcommand{\doubleinprod}[1]{\langle\!\langle{#1}\rangle\!\rangle}

\newcommand{\sump}{\sideset{}{'}\sum}
\newcommand{\SP}{\mathcal{F}}
\newcommand{\repSP}{\mathcal{F}}
\newcommand{\PL}{\mathcal P}
\newcommand{\domL}{{\mathbb T_\vecL}}

\newcommand{\bds}{\begin{displaystyle}}
\newcommand{\eds}{\end{displaystyle}}

\newcommand{\ddt}{\frac{d}{dt}}

\newcommand{\varep}{\varepsilon}
\newcommand{\s}{\sigma}

\newcommand{\widecheck}{\check}

\numberwithin{equation}{section}

\def\eqdef{\stackrel{\rm def}{=}}

\title{Asymptotic Expansions in Time for Rotating Incompressible Viscous Fluids}

\author{Luan T.~Hoang$^1$}
\address{$^1$ Department of Mathematics and Statistics, Texas Tech University,  1108 Memorial Circle, Lubbock, TX 79409-1042, U.S.A.}
\email{luan.hoang@ttu.edu}

\author{Edriss S.~Titi$^{2}$}
\address{$^2$ Department of Mathematics, 3368 TAMU, Texas A\&M University, College Station, TX 77843-3368, U.S.A.}
\address{Department of Applied Mathematics and Theoretical Physics, University of Cambridge, Cambridge CB3 0WA, U.K.}
\address{Weizmann Institute of Science, Department of Computer Science and Applied Mathematics,
P.O. Box 26, Rehovot, 76100, Israel}
\email{titi@math.tamu.edu and Edriss.Titi@damtp.cam.ac.uk}

\date{July 1, 2020}

\begin{document}

\newtheorem{theorem}{Theorem}[section]
\newtheorem{problem}[theorem]{Problem}
\newtheorem{lemma}[theorem]{Lemma}
\newtheorem{corollary}[theorem]{Corollary}
\newtheorem{proposition}[theorem]{Proposition}
\newtheorem{definition}[theorem]{Definition}
\newtheorem{assumption}[theorem]{Assumption}

\theoremstyle{remark}
\newtheorem{remark}[theorem]{\bf{Remark}}
\newtheorem{example}[theorem]{\bf{Example}}
\newtheorem{defn}[theorem]{\bf{Definition}}

\newcommand{\statenum}{\refstepcounter{equation}
{\flushleft (\theequation) } }

\begin{abstract}
We study the three-dimensional Navier--Stokes equations of rotating incompressible viscous fluids with periodic boundary conditions.
The asymptotic expansions,  as time goes to infinity, are derived in all Gevrey spaces for any Leray-Hopf weak solutions in terms of oscillating, exponentially decaying functions.
The results are established for all non-zero rotation speeds, and for both cases with  and without the zero spatial average of the solutions.
Our method makes use of the Poincar\'e waves to rewrite the equations, and then implements the Gevrey norm techniques to deal with the resulting time-dependent bi-linear form.
Special solutions are also found which form infinite dimensional invariant linear manifolds.
\end{abstract}

\keywords{Navier-Stokes equations, rotating fluids, asymptotic expansions, long-time dynamics.}

\subjclass[2010]{35Q30, 76D05, 35C20, 76E07.}

\maketitle

\tableofcontents
\pagestyle{myheadings}\markboth{L. T. Hoang and E. S. Titi}{Asymptotic Expansions for Rotating Incompressible Viscous Fluids}

\section{Introduction}\label{introsec}

We study the long-time behavior of the three-dimensional incompressible viscous fluids rotated about the vertical axis with a constant angular speed.
The Navier--Stokes equations (NSE) written in the rotating frame are used to describe the fluid dynamics in this case, see, e.g., \cite{VanyoBook2001}.
We denote by $\vecx\in\R^3$ the spatial variables, $t\in\R_+=[0,\infty)$ the time variable, and  $\{\vece_1,\vece_2,\vece_3\}$ the standard canonical basis of $\R^3$.
The NSE for the rotating fluids  are
\begin{align}
\label{NPDE}
\frac{\partial \vecu}{\partial t} - \nu \Delta \vecu + (\vecu\cdot\nabla) \vecu + \nabla p +
\Omega \vece_3\times \vecu &=0,\\
\label{divfree}
{\rm div}\, \vecu&=0,
\end{align}
where $\vecu(\vecx,t)$ is the velocity field, $p$ is the pressure adjusted by the fluid's constant density, gravity and centrifugal force,
 $\nu>0$ is the kinematic viscosity, and $\frac{1}{2}\Omega \vece_3$ is the angular velocity of the rotation.

Above, $\Omega \vece_3\times \vecu$ represents the Coriolis force exerted on the fluid. We will write
\beqs
\vece_3\times \vecu=J \vecu,\text{ where }
J=\begin{pmatrix}
  0&-1&0\\1&0&0\\0&0&0
 \end{pmatrix}.
\eeqs

Equations \eqref{NPDE} and \eqref{divfree} comprise a system of nonlinear partial differential equations with the unknowns $\vecu$ and $p$, while the constants $\nu>0$ and $\Omega$ are given. We will study this system within the context of spatially periodic functions.

Let $L_1,L_2,L_3>0$ be the spatial periods and denote $\vecL=(L_1,L_2,L_3)$.  A function $f:\R^3\to\R^m$, for some $m\in\N$, is $\vecL$-periodic if
\beqs
f(\vecx+L_j \vece_j)=f(\vecx)\text{ for all } \vecx\in\R^3,\ j=1,2,3.
\eeqs

Consider the $\vecL$-periodic solutions $(\vecu,p)$, that is, $\vecu(\cdot,t)$ and $p(\cdot,t)$ are $\vecL$-periodic for all $t> 0$.

Let $L_*=\max\{L_1,L_2,L_3\}$ and $\lambda_1=(2\pi/L_*)^2$.
Under the transformation
\beqs \vecu(\vecx,t)=\lambda_1^{1/2}\nu\, \vecv(\lambda_1^{1/2}\vecx,\lambda_1\nu
 t),
 \quad p(\vecx,t)=\lambda_1\nu^2 q(\lambda_1^{1/2}\vecx,\lambda_1\nu t),
 \quad \Omega=\lambda_1\nu\omega,
 \eeqs
where $\vecv(\vecy,\tau)$ and $q(\vecy,\tau)$ are $\lambda_1^{1/2}\vecL$-periodic adimensional functions,
system  \eqref{NPDE} and \eqref{divfree} becomes
\beqs \lambda_1^{3/2}\nu^2\Big(\frac{\partial \vecv}{\partial \tau} - \Delta_{\vecy} \vecv + (\vecv\cdot\nabla_{\vecy}) \vecv + \nabla_{\vecy} q + \omega J \vecv  \Big)=0\text{ and }
 \lambda_1\nu\, {\rm div}_{\vecy} \vecv=0,
\eeqs
thus,
\beq\label{vyeq}
\frac{\partial \vecv}{\partial \tau} - \Delta_{\vecy} \vecv + (\vecv\cdot\nabla_{\vecy}) \vecv + \nabla_{\vecy} q + \omega J \vecv =0
\text{ and }
{\rm div}_{\vecy} \vecv=0.
\eeq

Thanks to \eqref{vyeq},  we can assume hereafter, without loss of generality, that the Navier--Stokes system \eqref{NPDE} and \eqref{divfree} has
\beq\label{fixpar} \nu=1,\quad  L_*=2\pi,\quad  \lambda_1=1.
\eeq

In dealing with $\vecL$-periodic functions, it is convenient to formulate the equations and functional spaces using the domain (the three-dimensional flat torus)
$$\domL\eqdef (\R/L_1\Z)\times (\R/L_2\Z)\times (\R/L_3\Z).$$

Regarding the notation in this paper, a vector in $\C^3$ is viewed as a column vector, and we denote the dot product between $\vecx,\vecy\in\C^3$ by
\beqs
\vecx\cdot \vecy =\vecy^{\rm T}\vecx=\vecx^{\rm T}\vecy.
\eeqs
Hence, the standard inner product in $\C^3$ is $\vecx\cdot\bar\vecy$.

For $\veck=(k_1,k_2,k_3)\in\Z^3$, denote
\beq\label{kcheck}
\widecheck \veck=(\widecheck k_1,\widecheck k_2,\widecheck k_3)\eqdef 2\pi (k_1/L_1,k_2/L_2,k_3/L_3),
\eeq
and, in case $\veck\ne \mathbf 0$,
\beq\label{ktil}
\widetilde \veck=(\widetilde k_1,\widetilde k_2,\widetilde k_3)\eqdef \widecheck\veck/|\widecheck\veck|,
\eeq
\beq\label{Xk}
  X_\veck\eqdef \{ \vecz\in \C^3: \vecz\cdot\widecheck\veck=0 \}=\{ \vecz\in \C^3: \vecz\cdot\widetilde\veck=0 \}.
\eeq

In the following, we present the functional setting and functional formulations for the NSE. We refer the reader to the books \cite{LadyFlowbook69,CFbook,TemamAMSbook,TemamSIAMbook,FMRTbook} for more details.

Denote the inner product and norm in $L^2(\domL)^3$ by $\inprod{\cdot,\cdot}$ and $|\cdot|$, respectively.
The latter notation is also used for the modulus of a complex number and the length of a vector in $\C^n$, but its meaning will be clear from the context.

Each $\vecu\in L^2(\domL)^3$ has the Fourier series
\beq\label{uF}
\vecu(\vecx)=\sum_{\veck\in\Z^3} \widehat \vecu_\veck e^{i\widecheck\veck\cdot\vecx},
\eeq
where $i=\sqrt{-1}$, $\widehat\vecu_\veck\in \C^3$ are the Fourier coefficients with the reality condition $\widehat \vecu_{-\veck}= \overline{\widehat \vecu_\veck}$.
If $\vecu$ has zero spatial average over $\domL$ then $\widehat\vecu_{\mathbf 0}=0$.

We now focus on the case of solutions $\vecu(\vecx,t)$ with zero spatial average over $\domL$, for any $t\ge 0$. The general case will be studied in section \ref{nonzerosec}. For brevity, denote
\beqs
\sump=\sum_{\veck\in\Z^3\setminus\{\mathbf 0\}}.
\eeqs

Let $\mathcal V$ be the space of zero-average, divergence-free $\vecL$-periodic trigonometric polynomial vector fields, that is, it consists  of
functions
$$u=\sump  \widehat\vecu_\veck e^{i \widecheck \veck\cdot \vecx}$$
where  $\widehat\vecu_\veck\in X_\veck$,
 $\widehat\vecu_{-\veck}=\overline{\widehat\vecu_\veck}$ for all $\veck\in\Z^3\setminus\{\mathbf 0\}$,
and $\widehat\vecu_\veck\ne \mathbf 0$ for only finitely many $\veck$'s.

Let $H$, respectively (resp.) $V$, be the closure of $\mathcal V$ in $L^2(\domL)^3$, resp. $H^1(\domL)^3$.

We use the following embeddings and identification
$$V\subset H=H'\subset V',$$
where each space is dense in the next one, and the embeddings are compact.

Let $\PL$ denote the orthogonal (Leray) projection from $L^2(\domL)^3$ onto $H$.
More precisely,
\beqs
\PL \Big(\sum_{\veck\in\Z^3}  \widehat \vecu_\veck e^{i\widecheck\veck\cdot\vecx}\Big)=\sump [\widehat \vecu_\veck-(\widehat \vecu_\veck\cdot \widetilde\veck)\widetilde\veck] e^{i\widecheck\veck\cdot\vecx}
=\sum_{\veck\in\Z^3} \widehat P_\veck \widehat \vecu_\veck e^{i\widecheck\veck\cdot\vecx}
=\sump \widehat P_\veck \widehat \vecu_\veck e^{i\widecheck\veck\cdot\vecx},
\eeqs
where $\widehat P_\veck$'s are symmetric $3\times 3$ matrices given by
\beq\label{Pk}
\widehat P_{\mathbf 0}=0,\quad
\widehat P_\veck=I_3-\widetilde \veck \widetilde\veck^{\rm T}\text{ for }\veck\in\Z^3\setminus\{\mathbf 0\}.
\eeq

The Stokes operator $A$ is a bounded linear mapping from $V$ to its dual space $V'$ defined by
\beqs
\inprod{A\vecu,\vecv}_{V',V}=
\doubleinprod{\vecu,\vecv}
\eqdef \sum_{j=1}^3 \inprod{ \frac{\partial \vecu}{\partial x_j} , \frac{\partial \vecv}{\partial x_j} }, \text{ for all } \vecu,\vecv\in V.
\eeqs

As an unbounded operator on $H$, the operator $A$ has the domain $\mD(A)=V\cap H^2(\domL)^3$, and, under  the current consideration of periodicity conditions,
\beqs A\vecu = - \PL \Delta \vecu=-\Delta \vecu\in H, \text{ for all }\vecu\in\mD(A).
\eeqs
With the Fourier series, this reads as
\beqs
Au=\sump |\widecheck \veck|^2 \widehat\vecu_\veck e^{i \widecheck \veck\cdot \vecx},
\text{ for }
u=\sump \widehat\vecu_\veck e^{i \widecheck \veck\cdot \vecx}\in\mathcal D(A).
\eeqs

The spectrum  of $A$ is known to be
\beq\label{spA}
\sigma(A)=\big\{|\widecheck\veck|^2: \ \veck\in\Z^3, \veck\ne \mathbf 0\big\},
\eeq
and each $\lambda\in \sigma(A)$ is an eigenvalue with finite multiplicity.
We order
\beq\label{Biglam}
\sigma(A)=\{\Lambda_n:n\in\N\},\text{ where the sequence $(\Lambda_n)_{n=1}^\infty$ is strictly increasing}.
\eeq

The additive semigroup generated by $\sigma(A)$ is
\beq\label{semispecA}
\langle \sigma(A)\rangle \eqdef \Big\{\sum_{j=1}^n \alpha_j \,:\, n\in\N, \alpha_j\in\sigma(A)\text{ for } 1\le j\le n\Big\}.
\eeq

The set $\langle \sigma(A)\rangle$ is ordered as a strictly increasing sequence $(\mu_n)_{n=1}^\infty$, i.e.,
\beq\label{mudef}
\langle \sigma(A)\rangle=\{\mu_n:n\in\N\},\quad \mu_{n+1}>\mu_n \quad\forall n\in\N.
\eeq

Note, by \eqref{fixpar}, \eqref{spA}, \eqref{Biglam} and \eqref{mudef}, that $\mu_1=\Lambda_1=\lambda_1=1$.
Also,
\beq\label{musum}
\mu_m+\mu_n\in \langle \sigma(A)\rangle \quad\forall m,n\in\N,
\eeq
\beq\label{speclim}
\Lambda_n\to\infty \text{ and } \mu_n\to\infty \text{ as }n\to\infty.
\eeq

For $\Lambda\in\sigma (A)$, we denote by  $R_\Lambda$ the orthogonal projection from $H$ onto the eigenspace of $A$ corresponding to $\Lambda$,
and set $$P_\Lambda=\sum_{\lambda\in \sigma(A),\lambda\le \Lambda}R_\lambda.$$
Note that each vector space $P_\Lambda H$ is finite dimensional.

\medskip
For $\alpha,\sigma \in \R$ and  $u=\sump \widehat \vecu_\veck e^{i\widecheck\veck\cdot \vecx}\in H$, define
$$A^\alpha u=\sump |\widecheck\veck|^{2\alpha} \widehat \vecu_\veck e^{i\widecheck\veck\cdot
\vecx},\quad e^{\sigma A^{1/2}} u=\sump e^{\sigma
|\widecheck\veck|} \widehat \vecu_\veck e^{i\widecheck\veck\cdot
\vecx},$$
and
$$A^\alpha e^{\sigma A^{1/2}} u=\sump |\widecheck\veck|^{2\alpha}e^{\sigma
|\widecheck\veck|} \widehat \vecu_\veck e^{i\widecheck\veck\cdot
\vecx}.$$

For $\alpha,\sigma\ge 0$, the  Gevrey spaces are defined by
\beqs
G_{\alpha,\sigma}=\mD(A^\alpha e^{\sigma A^{1/2}} )\eqdef \{ u\in H: |u|_{\alpha,\sigma}\eqdef |A^\alpha
e^{\sigma A^{1/2}}u|<\infty\}.
\eeqs
In particular, when $\sigma=0$  the domain of the fractional operator $A^\al$ is
\beqs
\mD(A^\alpha)=G_{\alpha,0}=\{ u\in H: |A^\alpha u|=|u|_{\alpha,0}<\infty\}.
\eeqs

Thanks to the zero-average condition, the norm $|A^{m/2}u|$ is equivalent to $\|u\|_{H^m(\Omega)^3}$ on the space $\mathcal D(A^{m/2})$, for $m=0,1,2,\ldots$

Note that $\mD(A^0)=H$,  $\mD(A^{1/2})=V$, and  $\|u\|\eqdef |\nabla u|$ is equal to $|A^{1/2}u|$, for all $u\in V$.
Also, the norms $|\cdot|_{\alpha,\sigma}$ are increasing in $\alpha$, $\sigma$, hence, the spaces $G_{\alpha,\sigma}$ are decreasing in $\alpha$, $\sigma$.

Observe, for $\sigma>0$, that  $G_{\alpha,\sigma}$ consists of real analytic divergence-free vector fields and its norm is stronger than all the Sobolev norms. Any asymptotic expansions in such Gevrey spaces induce estimates of the remainders in their very strong norms, see \eqref{vremain} and \eqref{uremain} below. Those estimates can be beneficial to numericals, and have potential applications in studying practical aspects of the physics of fluids. See, e.g., \cite{DoeringTiti} for a consequence of the estimates in Gevrey norms connected with the Kolmogorov theory of turbulence.
From the technical point of view, the use of the Gevrey norms enables an immediate regularity bootstrapping in all Sobolev spaces via the method by Foias and Temam \cite{FT-Gevrey}.
It avoids the lengthy, gradual bootstrapping of step-by-step increment in Sobolev regularity. For example, the original paper by Foias and Saut \cite{FS87} has to establish that the time derivatives of all orders of the solution and remainders belong to all Sobolev spaces.
See also the contrast between the two bootstrapping mechanisms in the proof of Proposition 3.4 of \cite{HM2}.

\medskip
Regarding the nonlinear terms in the NSE, a bounded linear map $B:V\times V\to V'$ is defined by
\beqs
\inprod{B(\vecu,\vecv),\vecw}_{V',V}=b(\vecu,\vecv,\vecw)\eqdef \int_\domL ((\vecu\cdot \nabla) \vecv)\cdot \vecw\, d\vecx, \text{ for all } \vecu,\vecv,\vecw\in V.
\eeqs
In particular,
\beq\label{defBuv}
B(\vecu,\vecv)=\PL ((\vecu\cdot \nabla) \vecv), \text{ for all }\vecu,\vecv\in\mD(A).
\eeq
In fact, if $u=\sump \widehat\vecu_\veck  e^{i \widecheck \veck\cdot \vecx}\in \mathcal D(A)$ and
$v=\sump \widehat\vecv_\veck  e^{i \widecheck \veck\cdot \vecx}\in \mathcal D(A)$, then $(u\cdot \nabla )v$ has zero spatial average and
\begin{align}\label{unabv}
(u\cdot\nabla)v&=\sump \widehat\vecb_\veck  e^{i \widecheck \veck\cdot \vecx},\text{ where }
\widehat\vecb_\veck = \sum_{\widecheck\vecm+\widecheck j=\widecheck\veck} i (\widehat\vecu_\vecm\cdot \widecheck \veck) \widehat\vecv_\vecj, \\
\intertext{and, consequently, }
\label{BF}
B(u,v)&=\sump \widehat P_\veck \widehat\vecb_\veck  e^{i \widecheck \veck\cdot \vecx}.
\end{align}

Applying the projection $\PL$ of the equation \eqref{NPDE}, we obtain
\beq\label{ueq1}
\frac{du}{dt} +Au + B(u,u) +\Omega \PL J u =0,
\eeq
with solution $u\in H$.
In the case of non-rotation, i.e. $\Omega=0$, equation \eqref{ueq1} is the standard NSE
\beq\label{ueqzero}
\frac{du}{dt} +Au + B(u,u) =0.
\eeq

Since $\PL u=u$, for $u\in H$, equation \eqref{ueq1} is equivalent to
\beq\label{NSO}
\frac{du}{dt} +Au + B(u,u) +\Omega S u =0
\eeq
with $u\in H$, where $S=\PL J \PL$. Equation \eqref{NSO} will be the focus of our study.

Note that if $u$ is as in \eqref{uF}, then
\beq\label{SF}
Su=\sump \widehat P_\veck  J\widehat P_\veck \widehat \vecu_\veck e^{i\widecheck\veck\cdot\vecx}.
\eeq

We have the following elementary properties:
\beq\label{Sper}
\inprod{Su,A^\alpha e^{\sigma A^{1/2}} u}=0, \text{ for all } \alpha,\sigma\ge 0\text{ and } u\in \mathcal D(A^\alpha e^{\sigma A^{1/2}}),
\eeq
\beq\label{borth}
b(u,v,w)=-b(u,w,v) \text{ and } b(u,v,v)=0, \text{ for all } u,v,w\in V.
\eeq

Because of relation \eqref{Sper} with $\alpha=\sigma=0$, the energy balance/inequality for \eqref{NSO} is the same as for \eqref{ueqzero}.
Hence, the following definitions of  weak and regular solutions for \eqref{NSO} are quite similar to those for \eqref{ueqzero}, see, e.g., \cite{CFbook,FMRTbook,TemamAMSbook}.

\begin{definition}\label{lhdef}
{\rm (a)} A Leray-Hopf weak solution $u(t)$ of \eqref{NSO} is a mapping from $[0,\infty)$ to $H$ such that
\beqs
u\in C([0,\infty),H_{\rm w})\cap L^2_{\rm loc}([0,\infty),V),\
u'\in L^{4/3}_{\rm loc}([0,\infty),V'),
\eeqs
and satisfies
\beqs
\ddt \inprod{u(t),w}+\doubleinprod{u(t),w}+b(u(t),u(t),w)+\Omega \inprod{Su,w}=0
\eeqs
in the distribution sense in $(0,\infty)$ (in fact in $L^{4/3}_{\rm loc}([0,\infty))$), for all $w\in V$, and the energy inequality
\beqs
\frac12|u(t)|^2+\int_{t_0}^t \|u(\tau)\|^2d\tau\le \frac12|u(t_0)|^2
\eeqs
holds for $t_0=0$ and almost all $t_0\in(0,\infty)$, and for all $t\ge t_0$.
Here, $H_{\rm w}$ denotes the topological vector space $H$ with the weak topology.

{\rm (b)}  We say a function $u(t)$ is a Leray-Hopf weak solution on $[T,\infty)$ if $u(T+\cdot)$ is a Leray-Hopf weak solution.

A Leray-Hopf weak solution $u(t)$ on $[T_0,\infty)$, for some $T_0\in\R_+$, is a \emph{regular solution} on $[T_0,T_0+T)$, for some $0<T\le \infty$, if
\beqs
u\in C([T_0,T_0+T),V)\cap L^2_{\rm loc}([T_0,T_0+T),\mD(A)), \text{ and }u'\in L^2_{\rm loc}([T_0,T_0+T),H).
\eeqs

A \emph{global regular solution} is a regular solution on $[0,\infty)$.
\end{definition}

Same as for equation \eqref{ueqzero}, the basic questions, for equation \eqref{NSO}, about the uniqueness of weak solutions, and the global existence of regular solutions are still open, see  Theorem \ref{typicalthm} below.
Regarding the second question, it is proved in Babin-Mahalov-Nicolaenko \cite{BMN99b} that for any initial data $u_0\in V$, there is $\Omega_0>0$ depending on $u_0$ such that global regular solutions exist for all $|\Omega|>\Omega_0$.
Moreover, it is also showed in \cite{BMN99b} that the long-time dynamics of \eqref{NSO}, with an additional non-potential body force, is close to that of a so-called ``$2\frac12$-dimensional NSE'' when $|\Omega|$ is sufficiently large.
However, these basic questions will be bypassed and the mentioned results of \cite{BMN99b} will not be needed in this study.
It is due to the eventual regularity and decay (to zero) of the solutions, see Proposition \ref{decay2} below.
We, instead, will focus on the refined analysis of that decay.

When $\Omega=0$, the long-time dynamics of equation \eqref{ueqzero}, which covers the case of potential forces, is studied in details early in \cite{FS83,FS84a,FS84b,FS87,FS91}, and later in \cite{FHOZ1,FHOZ2,FHN1,FHN2,FHS1,HM1}. (The case of non-potential forces is treated in \cite{HM2,CaH1,CaH2}. See also the survey paper \cite{FHS2} for more information.) Briefly speaking, it is proved in \cite{FS87} that any solution $u(t)$ of \eqref{ueqzero} admits an asymptotic expansion, as $t\to\infty$,
\beq\label{exreal}
u(t)\sim \sum_{n=1}^\infty q_n(t)e^{-\mu_n t},
\eeq
where $q_n(t)$ is an $\mathcal V$-valued polynomial in $t$. See Definition \ref{Sexpand} below for more information.

When $\Omega\ne 0$, one can initially view equation \eqref{NSO} the same as \eqref{ueqzero} with the linear operator $\tilde A=A+\Omega S$ replacing $A$, and follow \cite{FS87} to obtain the asymptotic expansions. However, the spectrum of $\tilde A$, and its eigenspaces will be more complicated. The expansions will be of the form
\beq\label{excomplex}
u(t)\sim \sum_{\mu\in\langle \sigma(\tilde A)\rangle} q_\mu(t)e^{-\mu t},
\eeq
where $\langle \sigma(\tilde A)\rangle$ is defined similarly to \eqref{semispecA}.

One can see that the additive semigroup $\langle \sigma(\tilde A)\rangle$ is a set of complex numbers, much more complicated than $\langle \sigma(A)\rangle$, and consequently expansion \eqref{excomplex} is considerably more complicated than \eqref{exreal}. The construction of $q_\mu(t)$ must consider various scenarios including resonance and non-resonance cases, which are harder to track for complex values $\mu$'s.
To avoid all these technicalities, we propose another approach that converts \eqref{NSO} to \eqref{ueqzero} with a time-dependent bilinear form.
By maintaining the operator $A$, the expansions will be similar to \eqref{exreal}, and the proof will be direct and clean in the spirit of \cite{FS87}, taking advantage of recent improvements in \cite{HM1,HM2}. This approach was, in fact, successfully used in \cite{BMN99b} in studying the global well-posedness for regular solutions of \eqref{NSO}.

\medskip
\noindent\textbf{Rewriting the variational NSE using the Poincar\'e waves.}
In dealing with term $\Omega S u$ in \eqref{NSO}, we will make a change of variables using the exponential operator $e^{t S}$.
Clearly, $S$ is a bounded linear operator on $L^2(\domL)^3$ with norm $\|S\|_{\mathcal L(L^2(\domL)^3)}=1$.
For our problem, we restrict $e^{tS}$ to $H$ only, and have the isometry group $e^{tS}:H\to H$, $t\in\R$, is analytic in $t$.

It is well-known, see e.g. \cite{BMN99b,CheminBook2006},  that one has, for $u=\sump  \widehat\vecu_\veck e^{i\widecheck\veck\cdot\vecx}\in H$,
\beq\label{etS}
e^{tS}u=\sump  E_\veck(\widetilde k_3 t)\widehat\vecu_\veck e^{i\widecheck\veck\cdot\vecx},
\quad \text{where }
E_\veck(t)= \cos(t)I_3+\sin(t)J_\veck,
\eeq
with $J_\veck$ being the $3\times 3$ matrix for which $J_\veck\vecz=\widetilde\veck \times \vecz$, for all $\vecz\in \C^3$.
Explicitly, the matrix $J_\veck$, for $\veck\in\Z^3\setminus\{\mathbf 0\}$, is
\beq\label{Jk}
J_\veck=\begin{pmatrix}
  0 & -\widetilde k_3 & \widetilde k_2\\
  \widetilde k_3   & 0 & - \widetilde k_1 \\
  - \widetilde k_2 &  \widetilde k_1 &0
 \end{pmatrix}.
\eeq

For the reader's convenience, we include  a simple proof of \eqref{etS} in Appendix \ref{apd}.

One sees that
\begin{align*}
(E_\veck(t))^*&=E_\veck(-t),\\
|E_\veck(t)\vecz|&=|\vecz|,\text{ for all } \vecz\in X_\veck.
\end{align*}

From these properties, we deduce
\begin{align}
\label{eaj}
(e^{tS})^*&=e^{-tS} \quad(\text{on } H),\\
\label{ESiso}
|e^{tS}u|_{\alpha,\sigma}&=|u|_{\alpha,\sigma},\text{ for all }\alpha,\sigma\ge 0 \text{ and } u\in\mathcal D(A^\alpha e^{\sigma A^{1/2}}).
\end{align}

Also, some relations between $S$ and $A$ are
\beq\label{SAcom}
ASu=SAu,\quad e^{tS}Au=Ae^{tS}u,\text{ for all }u\in\mathcal D(A)\text{ and } t\in\R.
\eeq

The bi-linear form in \eqref{NSO} will be transformed to a similar, but time-dependent one that we describe below.

Let $t,\Omega\in\R$. Define $b(t,\cdot,\cdot,\cdot):V^3\to\R$ and $b_\Omega(t,\cdot,\cdot,\cdot):V^3\to\R$ by
\beqs
b(t,u,v,w)=b(e^{-tS}u,e^{-tS}v,e^{-tS}w),\quad b_\Omega(t,u,v,w)=b(\Omega t,u,v,w),
\eeqs
for all $u,v,w\in V$. We then  define $B(t,\cdot,\cdot):V\times V\to V'$ by
\beq\label{defBt}
\inprod{B(t,u,v),w}_{V',V}=b(t,u,v,w), \text{ for all } u,v,w\in V.
\eeq
In particular, thanks to \eqref{defBuv} and \eqref{eaj},
\beq\label{defBtuv}
B(t,u,v)=e^{ t S}B(e^{- t S}u,e^{- t S}v), \text{ for all } u,v\in\mD(A).
\eeq

Define $B_\Omega(t,u,v)=B(\Omega t,u,v)$.

\medskip
We now rewrite equation \eqref{NSO} using the Poincar\'e waves $e^{-\Omega t S}w$, for $w\in H$ and $t\in\R$.

Let $u(t)$ be a solution of \eqref{NSO}. Set $v(t)=e^{\Omega t S}u(t)$, or equivalently, $u(t)=e^{-\Omega t S} v(t)$.

If $u\in C^1((0,\infty),H)\cap C((0,\infty),\mathcal D(A))$ then, with \eqref{SAcom} taken into account, $v$  solves
\beq\label{veq}
\frac{dv}{dt} +Av+B_\Omega(t,v,v)=0,\quad t>0.
\eeq

With this formulation, equation \eqref{veq} resembles more with \eqref{ueqzero} than \eqref{NSO}.
The difference between \eqref{veq} and \eqref{ueqzero} is the time-dependent bi-linear form $B_\Omega(t,\cdot,\cdot)$.
However, this bi-linear form turns out to possess many features similar to $B(\cdot,\cdot)$ itself.

For example, from \eqref{defBt} and \eqref{borth}, we have, for all $\Omega,t\in\R$ and $u,v,w\in V$, that
\beq\label{BS1}
\inprod{ B_\Omega(t,u,v),w}_{V',V}=-\inprod{ B_\Omega(t,u,w),v}_{V',V},
\eeq
and, consequently,
\beq\label{BS2}
\inprod{B_\Omega(t,u,v),v}_{V',V}=0.
\eeq

This prompts the following definition of weak solutions of \eqref{veq}.

\begin{definition}\label{lhv}
A Leray-Hopf weak solution $v(t)$ of \eqref{veq} is a mapping from $[0,\infty)$ to $H$ such that
\beqs
v\in C([0,\infty),H_{\rm w})\cap L^2_{\rm loc}([0,\infty),V),\quad v'\in L^{4/3}_{\rm loc}([0,\infty),V'),
\eeqs
and satisfies
\beqs
\ddt \inprod{v(t),w}+\doubleinprod{v(t),w}+b_\Omega(t,v(t),v(t),w)=0
\eeqs
in the distribution sense in $(0,\infty)$ (in fact in $L^{4/3}_{\rm loc}([0,\infty))$), for all $w\in V$, and the energy inequality
\beqs
\frac12|v(t)|^2+\int_{t_0}^t \|v(\tau)\|^2d\tau\le \frac12|v(t_0)|^2
\eeqs
holds for $t_0=0$ and almost all $t_0\in(0,\infty)$, and all $t\ge t_0$.

Other definitions in {\rm (b)} of Definition \ref{lhdef} are extended to the solution $v(t)$.
\end{definition}

Similar to the case $\Omega=0$, we have the following existence, uniqueness and regularity results for equations \eqref{NSO} and \eqref{veq}.

\begin{theorem}\label{typicalthm}
Let $\Omega\ne 0$ be a given number.

\begin{enumerate}[label=\rm(\roman*)]
\item For any $u_0\in H$, there exists a Leray-Hopf weak solution $u(t)$ of \eqref{NSO}, resp. $v(t)$ of \eqref{veq}, with initial data $u_0$.
Moreover, there is $T_0=T_0(u_0)\ge 0$ such that $u$, resp. $v$, is a regular solution on $[T_0,\infty)$.

\item  For any $u_0\in V$, there exists a unique regular solution $u(t)$ of \eqref{NSO}, resp. $v(t)$ of \eqref{veq}, with initial data $u_0$,  on an interval $[0,T)$ for some $T>0$. If $\|u_0\|$ is sufficiently small, then $T=\infty$.

\item For any Leray-Hopf weak solution $u(t)$ of \eqref{NSO}, resp. $v(t)$ of \eqref{veq}, and any number $\sigma>0$, there exists $T_*>0$ such that $u(t)$, resp. $v(t)$, belongs to $G_{1/2,\sigma}$, for all $t\ge T_*$, and the equation \eqref{NSO}, resp. \eqref{veq}, holds in  $\mathcal D(A)$ on $(T_*,\infty)$ with classical time derivative.
\end{enumerate}
\end{theorem}

Parts (i) and (ii) of Theorem \ref{typicalthm} are standard. Part (iii) can be proved by using the same technique of Foias-Temam \cite{FT-Gevrey}, noticing that we have the orthogonality \eqref{Sper}. See details of similar calculations in \cite{HM1}, and also more specific statements in Proposition \ref{decay2} below.

The current paper is focused on a different question, namely, the precise long-time dynamics of the solutions of \eqref{NSO} for all $\Omega\ne 0$. Even though they eventually, as $t\to\infty$, go to zero, our goal is to provide a detailed description of such a decay. In the case with the zero spatial average of the solutions, we will show that each solution possesses an asymptotic expansion which is similar to \eqref{exreal}, but contains some oscillating terms that are from the rotation.
The oscillating parts, in this case, are written in terms of sinusoidal functions of time. (See a similar result by Shi \cite{Shi2000} for dissipative wave equations.)
In the general case of non-zero spatial average, similar asymptotic expansions are obtained with the oscillating terms being expressed by the ``double sinusoidal'' functions.

This paper is organized as follows.
In section \ref{expsec}, we introduce our abstract asymptotic expansions, for large time, in terms of oscillating-decaying functions.
Our basic classes of functions are the S-polynomials and SS-polynomials, see Definition \ref{polyS}, below. Fundamental properties of these two classes are studied. In particular, when the forcing term of a linear ordinary differential equation (ODE) is a perturbation of an S-polynomial, then the  ODE's solution can be approximated by S-polynomials, see Lemma \ref{polylem}, below. This lemma turns out to be a building block in our constructions of polynomials in the asymptotic expansions.
In section \ref{zerosec}, we establish the asymptotic expansions for solutions of \eqref{veq} and \eqref{NSO}, below, in Theorems \ref{vthm} and \ref{uthm}, respectively. The expansions are in terms of S-polynomials and exponential functions. It is worth mentioning that these results are established for \emph{all} $\Omega\ne 0$ and for all Leray-Hopf weak solutions.
In particular, it does not rely on Babin-Mahalov-Nicolaenko's global well-posedness result \cite{BMN99b}, which, as mentioned after Definition \ref{lhdef}, requires $\Omega$  to be \emph{sufficiently large} depending on the initial data.
In section \ref{nonzerosec}, we derive the asymptotic expansions for solutions without the zero spatial averages.
This is done by using the specific Galilean-type transformation \eqref{Gtrans}. Unlike the previous section, this yields the expansions in terms of SS-polynomials.
In section \ref{specialsec}, we present some special solutions that form infinite dimensional linear manifolds that are invariant under the flows generated by the solutions of \eqref{NSO}. Especially, Remark \ref{heli} contains a subclass of solutions for which the helicity, a meaningful physical quantity in fluid dynamics \cite{Moffatt69,Moffatt92}, vanishes. The case of non-zero spatial average of the solutions is treated in Theorem \ref{specthm2}, below.

\section{Abstract asymptotic expansions and their properties}
\label{expsec}

We introduce here the classes of functions which will appear in our asymptotic expansions.

\begin{definition}\label{polyS}
Let $X$ be a vector space over the scalar field $\K=\R$ or $\K=\C$.
\begin{enumerate}[label=\rm(\alph*)]
 \item A function $g:\R\to X$ is an $X$-valued S-polynomial if it is a finite sum of the functions in the set
\beqs
\Big \{ t^m \cos(\omega t)Z,\ t^m \sin(\omega t) Z: m\in\N\cup\{0\},\ \omega\in\R, \ Z\in X\Big\}.
\eeqs

\item A function $g:\R\to X$ is an $X$-valued SS-polynomial if it is a finite sum of the functions in the set
\begin{align*}
&\Big\{t^m \cos\big( a\cos(\omega t)+b\sin(\omega t)+ct+d\big)Z,\\
&\quad t^m \sin\big( a\cos(\omega t)+b\sin(\omega t)+ct+d\big)Z:\\
&\quad m\in\N\cup\{0\},\ a,b,c,d,\omega\in \R,\  Z\in X\Big\}.
\end{align*}

\item Denote by $\SP_0(X)$, resp., $\SP_1(X)$ and $\SP_2(X)$ the set of all $X$-valued polynomials, resp., S-polynomials and SS-polynomials.
\end{enumerate}
\end{definition}

Clearly, $\SP_0(X)$, $\SP_1(X)$ and $\SP_2(X)$ are vector spaces over $\K$, and $\SP_0(X)\subset \SP_1(X)\subset \SP_2(X)$.

If $f\in \SP_1(X)$, then we can write $f$ as
\beq\label{fpow}
f(t)=\sum_{n=0}^N t^n f_n(t),
\eeq
where
\beq\label{fcos}
f_n(t)=\sum_{j=1}^{N_n} [a_{n,j}\cos(\omega_{n,j} t)+b_{n,j}\sin(\omega_{n,j}t)],
\eeq
with $N_n\in \N$, $a_{n,j},b_{n,j}\in X$, and $\omega_{n,j}\ge 0$ with the mapping $j\mapsto \omega_{n,j}$ being strictly increasing for each fixed $n$.

\begin{definition}\label{Sexpand}
Let $(X,\|\cdot\|_X)$ be a normed space and  $(\alpha_n)_{n=1}^\infty$ be a  sequence of strictly increasing non-negative real numbers.
Let $\repSP=\SP_0$, $\SP_1$, or $\SP_2$.
A function $f:[T,\infty)\to X$, for some $T\in\R_+$, is said to have an asymptotic expansion
\beq\label{ex1}
f(t)\sim\sum_{n=1}^\infty f_n(t)e^{-\alpha_n t} \quad\text{in } X,
\eeq
where each $f_n$ belongs to $\repSP(X)$, if one has, for any $N\ge 1$, that
\beq\label{ex2}
\Big \|f(t)-\sum_{n=1}^N f_n(t)e^{-\alpha_n t}\Big \|_X=\mathcal O(e^{-(\alpha_N+\varepsilon_N)t}),\text{ as }t\to\infty,
\eeq
for some $\varepsilon_N>0$.
\end{definition}

With this definition, the precise statement of \eqref{exreal} is that it holds, with $\SP(X)=\SP_0(X)$, in the space $X=G_{\alpha,\sigma}$, for all $\alpha,\sigma\ge 0$.

The expansion \eqref{ex1} with $\SP=\SP_1$ is equivalent to the one used in \cite{Shi2000}, for the dissipative wave equations, while with the largest class $\SP=\SP_2$ is new.
For other related asymptotic expansions for solutions of NSE, see \cite{CaH1,CaH2}.

One can observe, same as in \cite[Remark 2.3]{HM2}, that if \eqref{ex2} holds for all $N$ then
\beq\label{ex3}
\Big \|f(t)-\sum_{n=1}^N f_n(t)e^{-\alpha_n t}\Big \|_X=\mathcal O(e^{-\alpha t}),\text{ as }t\to\infty,
\eeq
for all $N$ and all $\alpha\in (\alpha_N,\alpha_{N+1})$.

\begin{lemma}\label{unilem}
 Given $N\in\N$, if $f_1,f_2,\ldots,f_N\in \SP_1(X)$ satisfy \eqref{ex2}, then such  S-polynomials $f_n$'s are unique.
\end{lemma}
\begin{proof}
Let $g_1,g_2,\ldots,g_N$ be functions in $\SP_1(X)$ that satisfy
\beqs
\Big \|f(t)-\sum_{n=1}^N g_n(t)e^{-\alpha_n t}\Big \|_X=\mathcal O(e^{-(\alpha_N+\varepsilon_N')t}),\text{ as }t\to\infty\text{ for some }\varep_N'>0.
\eeqs

Let $h_n=f_n-g_n$. By the triangle inequality, we have
\beqs
\Big \|\sum_{n=1}^N h_n(t)e^{-\alpha_n t}\Big \|_X
\le \Big \|f(t)-\sum_{n=1}^N f_n(t)e^{-\alpha_n t}\Big \|_X+\Big \|f(t)-\sum_{n=1}^N g_n(t)e^{-\alpha_n t}\Big \|_X,
\eeqs
hence
\beq\label{exh}
\Big \|\sum_{n=1}^N h_n(t)e^{-\alpha_n t}\Big \|_X=\mathcal O(e^{-(\alpha_N+\varepsilon)t}),\text{ as }t\to\infty,\text{ for }\varep=\min\{\varep_N,\varep_N'\}>0.
\eeq

Suppose that not all $h_n$'s are zero functions. Let $n_0$ be the smallest number such that $h_{n_0}\ne 0$.
By multiplying \eqref{exh} with $e^{\alpha_{n_0} t}$, we deduce
\beq \label{hnot}
\|h_{n_0}(t)\|_X=\mathcal O(e^{-\varepsilon' t}), \text{ for some }\varep'>0.
\eeq

We write $h_{n_0}(t)$ in the form of \eqref{fpow}.
Suppose the highest order term (with respect to the power of $t$) of $h_{n_0}(t)$ is $t^d z(t)$,
where $d$ is a non-negative integer, and $z$ is a non-zero function of the form as in the RHS of \eqref{fcos}.
Then dividing \eqref{hnot} by $t^d$ yields
\beq \label{zX}
\lim_{t\to\infty}\|z(t)\|_X=0.
\eeq

Suppose
\beq\label{zo}
z(t)=\sum_{n=1}^{N_d} [a_n\cos(\omega_n t)+b_n\sin(\omega_n t)],
\eeq
where $N_d\in\N$, $a_n,b_n\in X$ and $\omega_n\in\R$, for $1\le n\le N_d$.

Let $Y$ be the vector space spanned by $\{a_n,b_n\in X: 1\le n\le N_d \}$.
Let $\mathcal Y=\{Y_j: 1\le j\le m\}$ be a basis of $Y$.
By representing the vectors $a_n$'s and $b_n$'s in basis $\mathcal Y$, we can rewrite $z(t)$ as
\beqs
z(t)=\sum_{j=1}^m z_j(t) Y_j,
\eeqs
where $m\in\N$, each $z_j:\R\to \K$, for $1\le j\le m$, is a linear combination of the functions $\cos(\omega_n t)$ and $\sin(\omega_n t)$, for $1\le n\le N_d$, in \eqref{zo}.

For $y=\sum_{j=1}^m y_j Y_j\in Y$, with $y_j\in\K$ for $1\le j\le m$, define the norm
\beqs
\|y\|_Y=\Big(\sum_{j=1}^m |y_j|^2\Big)^{1/2}.
\eeqs

On the finite dimensional space $Y$, the norms $\|\cdot\|_X$ and $\|\cdot\|_Y$ are equivalent. Therefore, \eqref{zX} gives
\beqs
\lim_{t\to\infty}\|z(t)\|_Y=0, \text{ which implies }
\lim_{t\to\infty} z_j(t)=0, \text{ for } 1\le j\le m.
\eeqs

By Lemma \ref{uniR}, we obtain $z_j=0$ for all $j$. Hence $z=0$, which leads to a contradiction.
Thus, $h_n=0$, for all $n=1,2,\ldots,N$.  We conclude $f_n=g_n$, for $1\le n \le N$. The proof is complete.
\end{proof}

\begin{remark}\label{uniG}
We discuss a consequence of Lemma \ref{unilem}. Suppose \eqref{ex2} is satisfied with $X=G_{\beta_i,\sigma_i}$  and $f_n=g_n^{(i)}\in \SP_1(X)$, for $i=1,2$ and $1\le n\le N$.
Let $\bar\beta=\min\{\beta_1,\beta_2\}$ and $\bar\sigma=\min\{\sigma_1,\sigma_2\}$. Then \eqref{ex2} is satisfied with both $f_n=g_n^{(1)}$ and $f_n=g_n^{(2)}$ on the same space $X=G_{\bar\beta,\bar\sigma}$. By the virtue of Lemma \ref{unilem}, we have $g_n^{(1)}=g_n^{(2)}$, for $1\le n\le N$.
\end{remark}

\subsection{Properties of S-polynomials and SS-polynomials}
\label{ss21}

We start this subsection with some elementary properties of the functions introduced in Definition \ref{polyS} above.

\begin{lemma}\label{Spolyinvar}
Let $X$ and $\K$ be as in Definition \ref{polyS}, and $\repSP(X)=\SP_1(X)$ or $\repSP(X)=\SP_2(X)$.
Let $f$ be any function in $\repSP(X)$.
\begin{enumerate}[label={\rm(\roman*)}]
\item  The functions $t\mapsto f(T+t)$ and $t\mapsto f(k t)$ belong to $\repSP(X)$ for any numbers $T,k\in\R$.

\item If $g\in  \SP_1(\K)$, then $gf\in \repSP(X)$.

\item In case $X$ is a subspace of $\C^n$, with  $\K=\C$, then $e^{i\omega t}f\in \repSP(X).$

\item Suppose $Y$ is another vector space over $\K$ and $L$ is a linear mapping from $X$ to $Y$. Then
$Lf\in \repSP(Y)$.
\end{enumerate}
\end{lemma}
\begin{proof}
Parts (i) and (ii) can be easily verified by using elementary trigonometric identities such as the sine and cosine of a sum, and the product to sum formulas.
Part (iii) is obtained by applying part (ii) to the function $g(t)=e^{i\omega t}$ which, in fact, belongs to $\SP_1(\C)$.
Part (iv) is obvious.
\end{proof}

\begin{lemma}\label{SPlem1}
Let $\repSP=\SP_1$ or $\repSP=\SP_2$.
Then, a function $f$ belongs to $\repSP(\mathcal V)$ if and only if
\beq \label{fsum}
f(t)=\sump_{\text{\rm finitely many $\veck$}} \vecf_\veck(t)e^{i\widecheck\veck\cdot\vecx}, \text{ with each}\quad \vecf_\veck\in \repSP(X_\veck), \text{ and }
\vecf_{-\veck}=\overline{\vecf_\veck},
\eeq
where $X_\veck$ is as in \eqref{Xk}.
\end{lemma}
\begin{proof}
We prove for the case $\repSP=\SP_1$. The arguments for the other case $\repSP=\SP_2$ are similar and omitted.

Suppose $f\in \SP_1(\mathcal V)$. We write
\beqs
f(t)=\sum_{j=1}^N t^{m_j}[a_j \cos(\omega_j t)+b_j\sin(\omega_j t)]u_j, \text{ each } a_j,b_j\in\R, \ u_j\in\mathcal V.
\eeqs
By writing the finite Fourier series of each $u_j$ and combining the coefficients for $e^{i\widecheck\veck\cdot\vecx}$,
we find that the Fourier series of $f$ is of the form as in \eqref{fsum} with each $\vecf_\veck(t)$ being a finite sum of
$t^{m_j}[a_j \cos(\omega_j t)+b_j\sin(\omega_j t)] \vecz$ for some $\vecz\in X_\veck$. Thus, $\vecf_\veck\in \SP_1(X_\veck)$.
The last relation in \eqref{fsum} is the standard condition for $f$ to be real-valued.

Now, suppose $f$ is as in \eqref{fsum}. For each $\veck$, consider the function
$$F_\veck(t)=\vecf_\veck(t)e^{i\widecheck\veck\cdot\vecx}+\vecf_{-\veck}(t)e^{-i\widecheck\veck\cdot\vecx}
=\vecf_\veck(t)e^{i\widecheck\veck\cdot\vecx}+\overline{\vecf_\veck(t)}e^{-i\widecheck\veck\cdot\vecx}.$$

If  $\vecf_\veck(t)$ contains $t^m\cos(\omega t)\vecz$ for some $\vecz\in X_\veck$, then
$F_\veck(t)$ contains
 \beq \label{aterm}
 t^m\cos(\omega t) (\vecz e^{i\widecheck\veck\cdot\vecx} + \bar \vecz e^{-i\widecheck\veck\cdot\vecx}).
 \eeq
Since $\vecz e^{i\widecheck\veck\cdot\vecx} + \bar \vecz e^{-i\widecheck\veck\cdot\vecx}\in\mathcal V$, the function in \eqref{aterm} belongs to $\SP_1(\mathcal V)$.
Similar property holds for $\sin(\omega t)$ replacing $\cos(\omega t)$, and we obtain $F_\veck\in \SP_1(\mathcal V)$.
Then $f$ being a finite sum of such $F_\veck$'s yields $f\in \SP_1(\mathcal V)$.
\end{proof}

The following are important properties relating the S- and SS- polynomials with the rotation and nonlinear terms in the rotational NSE.

\begin{lemma}\label{SPVlem}
Let  $\Lambda\in \sigma(A)$, two functions $f,g\in \SP_1(P_\Lambda H)$, and $\Omega\in\R$.
Then
\beq\label{Sf1}
e^{\Omega tS}f(t)\in \SP_1(P_\Lambda H),\eeq
\beq\label{Sf2}
B(f(t),g(t))\in \SP_1(P_{4\Lambda} H),
\eeq
\beq\label{Sf3}
B_\Omega(t,f(t),g(t))\in \SP_1(P_{4\Lambda}H).
\eeq
\end{lemma}
\begin{proof}
(a) By Lemma \ref{SPlem1}, we can write $f(t)$ as
\beq \label{ft}
f(t)=\sump_{|\widecheck \veck|^2\le \Lambda} \vecf_\veck(t)e^{i\widecheck\veck\cdot\vecx}, \text{ with each } \vecf_\veck\in \SP_1(X_\veck), \text{ and }
\vecf_{-\veck}=\overline{\vecf_\veck}.
\eeq

Applying \eqref{etS} yields
\beq\label{eOf}
e^{\Omega t S} f(t)=\sump_{|\widecheck k|^2\le \Lambda} E_\veck(\Omega \widetilde k_3 t)\vecf_\veck(t) e^{i\widecheck\veck\cdot\vecx}.
\eeq
Applying Lemma \ref{Spolyinvar} to $\repSP=\SP_1$, $g(t):=\cos(\Omega \widetilde k_3 t)$ and then $g(t):=\sin(\Omega \widetilde k_3 t)$, one has that each $E_\veck(\Omega \widetilde k_3 t)\vecf_\veck(t)$ belongs to $\SP_1(X_\veck)$. Then by the virtue of the sufficient condition in Lemma \ref{SPlem1},  we have $e^{\Omega t S} f(t)\in \SP_1(\mathcal V)$. This and the restriction $|\widecheck\veck|^2\le \Lambda$ in \eqref{eOf} give \eqref{Sf1}.

(b) We prove \eqref{Sf2}. By Lemma \ref{SPlem1} again, we can assume, in addition to \eqref{ft}, that
\beqs
g(t)=\sump_{|\widecheck \veck|^2\le \Lambda} \vecg_\veck(t)e^{i\widecheck\veck\cdot\vecx}, \text{ with each } \vecg_\veck\in \SP_1(X_\veck), \text{ and }
\vecg_{-\veck}=\overline{\vecg_\veck}.
\eeqs

By \eqref{BF}, we have
\beq \label{Btsum}
B(f(t),g(t))=\sump \mathbf B_\veck(t) e^{i \widecheck \veck\cdot\vecx},
\eeq
where
\beq\label{Bk}
\mathbf B_\veck (t)
= \sum_{\stackrel{0<|\widecheck\vecm|^2,|\widecheck\vecj|^2\le \Lambda,}{\widecheck \vecm+\widecheck \vecj=\widecheck \veck}}
 (\vecf_\vecm(t)\cdot i\, \widecheck\veck)  \widehat P_\veck \vecg_\vecj(t).
\eeq

Thanks to formula \eqref{Bk} for $\mathbf B_\veck$, we only need to sum over $\veck$ in \eqref{Btsum} with
\beqs
|\widecheck \veck|^2=|\widecheck \vecm+\widecheck \vecj|^2\le 2(|\widecheck \vecm|^2+|\widecheck \vecj|^2)\le 4\Lambda.
\eeqs
Thus, $B(f(t),g(t))\in P_{4\Lambda}H$ for all $t\in\R$.

Note that $\widehat P_\veck =\widehat P_{-\veck} $.
In the sum \eqref{Bk}, we will pair $\widecheck \vecm+\widecheck \vecj=\widecheck \veck$ for $\mathbf B_\veck(t)$ with $(-\widecheck \vecm)+(-\widecheck \vecj)=(-\widecheck \veck)$ for $\mathbf B_{-\veck}(t)$.
Because $\big(\vecf_\vecm\cdot i\, \widecheck\veck\big) \widehat P_\veck\vecg_\vecj$, for $\mathbf B_\veck(t)$,
and
$ \big(\vecf_{-\vecm}\cdot  (-i\,\widecheck\veck)\big) \widehat P_{-\veck}\vecg_{-\vecj}$, for $\mathbf B_{-\veck}(t)$,
are conjugates of each other,  so are  $\mathbf B_\veck(t)$ and $\mathbf B_{-\veck}(t)$.

Using Lemma \ref{Spolyinvar} (ii) and (iv), one can verify that
$\big(\vecf_\vecm\cdot i\, \widecheck\veck\big) \widehat P_\veck \vecg_\vecj$ belongs to $\SP_1(X_\veck)$,
hence  $\mathbf B_\veck(t)\in \SP_1(X_\veck)$.

 By combining the above facts with Lemma \ref{SPlem1}, we conclude \eqref{Sf2}.

(c) Next, we prove \eqref{Sf3}. Because $f(t),g(t)\in \SP_1(P_\Lambda H)$, we apply \eqref{Sf1} to have $e^{-\Omega tS}f(t)$ and $e^{-\Omega tS}g(t)$ belong to $\SP_1(P_\Lambda H)$,
 which, by \eqref{Sf2}, imply
 $$B(e^{-\Omega tS}f(t),e^{-\Omega tS}g(t))\in \SP_1(P_{4\Lambda} H),$$
 which, in turn, thanks to \eqref{Sf1} again, further implies
 $$e^{\Omega tS} B(e^{-\Omega tS}f(t),e^{-\Omega tS}g(t))\in \SP_1(P_{4\Lambda} H).$$
Therefore, thanks also to \eqref{defBtuv}, we conclude \eqref{Sf3}.
 \end{proof}

\subsection{Approximating solutions of certain linear ODEs with S-polynomials}
\label{ss22}

In our proofs, we often need the following integrals.

\begin{lemma}\label{intlem}
Let $\alpha,\omega\in\R$ with $\alpha^2+\omega^2>0$, and $m$ be a non-negative integer.
Then each integral
 \beq\label{tei}
 \int t^m e^{\alpha t}\cos(\omega t) dt,\quad  \int t^m e^{\alpha t}\sin(\omega t) dt
 \eeq
is of the form $$p(t)e^{\alpha t}\cos(\omega t)+q(t) e^{\alpha t}\sin(\omega t)+const.,$$
where $p(t)$ and $q(t)$ are polynomials of degrees at most $m$.
\end{lemma}

Although this lemma is elementary, we give a proof in Appendix \ref{apd} that yields simple and explicit formulas for the integrals in \eqref{tei}, see \eqref{tmI} below.

The next lemma essentially originates from Foias-Saut \cite{FS87}, but is stated and proved in the same convenient form as \cite[Lemma 4.2]{HM2}.

\begin{lemma}\label{polylem}
Let $(X,\|\cdot\|_X)$ be a Banach space. Suppose $y$ is a function in $C([0,\infty),X)$, with distribution derivative $y'\in L^1_{\rm loc}([0,\infty),X)$, that solves the following ODE
 \beqs
 y'(t)+ \beta y(t) =p(t)+g(t)
 \eeqs
 in the $X$-valued distribution sense on $(0,\infty)$, where  $\beta\in \R$ is a constant, $p(t)$ is an $X$-valued S-polynomial, and $g\in L^1([0,\infty),X)$ satisfies
 \beq\label{gM}
 \|g(t)\|_X\le Me^{-\delta t}, \text{ for all } t\ge 0,  \text{ and some } M,\delta>0.
 \eeq

 Define $q(t)$, for $t\in  \R$, by
\beq\label{qdef}
 q(t)=
\begin{cases}
e^{-\beta t}\int_{-\infty}^t e^{\beta\tau }p(\tau) d\tau&\text{if }\beta >0,\\
y(0) +\int_0^\infty  g(\tau)d\tau + \int_0^t p(\tau)d\tau &\text{if }\beta =0,\\
-e^{-\beta t}\int_t^\infty e^{\beta\tau }p(\tau) d\tau&\text{if }\beta <0.
\end{cases}
\eeq

Then $q(t)$ is an $X$-valued S-polynomial  that  satisfies
\beq\label{pode2}
q'(t)+\beta q(t) = p(t),\text{ for all } t\in \R,
\eeq
and the following estimates hold:

\begin{enumerate}[label={\rm (\roman*)}]
  \item
  If $\beta>0$ then
  \beq\label{zsquare}
\|y(t)-q(t)\|_X^2 \le 2e^{-2\beta t}\|y(0)-q(0)\|_X^2 + 2t \int_0^t e^{-2\beta(t-\tau)} \|g(\tau)\|_X^2 d\tau,
\text{ for all } t\ge 0.
\eeq

  \item If either

  {\rm (a)} $\beta=0$, or

  {\rm (b)}  $\beta<0$ and
  \beq\label{limycond} \lim_{t\to\infty} (e^{\beta t}\|y(t)\|_X)=0,\eeq
then
  \beq\label{g1b3}
  \|y(t)-q(t)\|_X^2\le \Big(\frac{M}{\delta-\beta}\Big)^2e^{-2\delta t}, \quad \text{ for all } t\ge 0.
  \eeq
  \end{enumerate}
\end{lemma}
\begin{proof}
Thanks to Lemma \ref{intlem}, $q(t)$ is an $X$-valued S-polynomial. The rest of this lemma is the same as \cite[Lemma 4.2]{HM2}, except for the relaxed estimate \eqref{zsquare} which we verify now.
Consider $\beta>0$. Let $z(t)=y(t)-q(t)$. Recall the inequality after (4.13) in \cite[Lemma 4.2]{HM2}, for all $t\ge 0$,
\beqs
\|z(t)\|_X\le e^{-\beta t} \|z(0)\|+ \int_{t_0}^t e^{-\beta(t-\tau)}\|g(\tau)\|_X d\tau.
\eeqs

Using Cauchy-Schwarz's and H\"older's inequalities, we estimate
\begin{align*}
\|z(t)\|_X^2
&\le 2e^{-2\beta t}\|z(0)\|_X^2+ 2\Big( \int_0^t e^{-\beta (t-\tau)}\|g(\tau)\|_X d\tau\Big)^2\\
& \le 2e^{-2\beta t}\|z(0)\|_X^2 + 2t \int_0^t e^{-2\beta(t-\tau)} \|g(\tau)\|_X^2 d\tau.
\end{align*}
Therefore, we obtain \eqref{zsquare}.
\end{proof}

\section{The case of zero spatial average solutions}\label{zerosec}

We obtain two main asymptotic expansion results, one for equation \eqref{veq}, and the other for equation \eqref{NSO}.
We recall that the numbers $\mu_n$'s are defined in \eqref{mudef} with basic properties \eqref{musum} and \eqref{speclim}.

\begin{theorem}
\label{vthm}
For any Leray-Hopf weak solution $v(t)$ of \eqref{veq}, there exist unique $\mathcal V$-valued S-polynomials $q_n$'s, for all $n\in\N$, such that it holds, for any $\alpha,\sigma >0$ and $N\ge 1$, that
\beq\label{vremain}
\Big|v(t)-\sum_{n=1}^N q_n(t)e^{-\mu_nt}\Big|_{\alpha,\sigma}= \mathcal O
\big(e^{-\mu t}\big),\text{ as } t\to\infty,\text{ for all }\mu\in(\mu_N,\mu_{N+1}).
\eeq
That is,
\beqs
v(t)\sim \sum_{n=1}^\infty q_n(t)e^{-\mu_nt}\quad \text{ in } G_{\alpha,\sigma}, \text{ for all } \alpha,\sigma>0.
\eeqs
\end{theorem}

Theorem \ref{vthm} is our key technical result. With this, we immediately obtain the asymptotic expansions for solutions of \eqref{NSO}.

\begin{theorem}
\label{uthm}
Let $u(t)$ be any Leray-Hopf weak solution of \eqref{NSO}. Then there exist unique $\mathcal V$-valued S-polynomials $Q_n$'s, for all $n\in\N$,  such that
it holds, for any $\alpha,\sigma >0$ and $N\ge 1$, that
\beq\label{uremain}
\Big|u(t)-\sum_{n=1}^N Q_n(t)e^{-\mu_n t} \Big|_{\alpha,\sigma}= \mathcal O
\big(e^{-\mu t}\big),\text{ as } t\to\infty,\text{ for all }\mu\in(\mu_N,\mu_{N+1}).
\eeq
That is,
\beqs
u(t)\sim \sum_{n=1}^\infty Q_n(t)e^{-\mu_nt}\quad \text{ in } G_{\alpha,\sigma}, \text{ for all } \alpha,\sigma>0.
\eeqs

Moreover, each $Q_n$ is related to $q_n$ in Theorem \ref{vthm} via relation \eqref{e1}, below.
\end{theorem}
\begin{proof}
Let $T_*>0$ be as in Theorem \ref{typicalthm}(iii).
Set $v(t)=e^{\Omega t S} u(t)$. Then $v(t)$ is a regular solution of \eqref{veq} on $[T_*,\infty)$.
Applying Theorem \ref{vthm} to solution $v(T_*+t)$, we have, for all $\alpha,\sigma>0$ and $N\ge 1$, that
\beq\label{vremT}
\Big|v(T_*+t)-\sum_{n=1}^N q_n(t)e^{-\mu_nt}\Big|_{\alpha,\sigma}= \mathcal O
\big(e^{-\mu t}\big),\text{ as }  t\to\infty,\text{ for all }\mu\in(\mu_N,\mu_{N+1}),
\eeq
where all $q_n$'s belong to $\SP_1(\mathcal V)$.
By shifting the time variable, we obtain from \eqref{vremT} that
\beq\label{vrem2}
\Big|v(t)-\sum_{n=1}^N q_n(t-T_*)e^{-\mu_n (t-T_*)}\Big|_{\alpha,\sigma}= \mathcal O
\big(e^{-\mu t}\big), \text{ as } t\to\infty.
\eeq

Let
\beq\label{e1} Q_n(t)=e^{\mu_n T_*} e^{-\Omega t S} q_n(t-T_*).
\eeq
Rewrite the left-hand side of \eqref{vrem2} as
\beqs
\Big|e^{\Omega t S} \Big(u(t)-\sum_{n=1}^N Q_n(t) e^{-\mu_nt}\Big)\Big|_{\alpha,\sigma},
\text{ which equals }\Big|u(t)-\sum_{n=1}^N Q_n(t)e^{-\mu_n t}\Big|_{\alpha,\sigma}
\eeqs
thanks to the isometry \eqref{ESiso}. Thus, we obtain \eqref{uremain}.
Thanks to Lemma \ref{Spolyinvar}(i), each $q_n(t-T_*)$ is a  $\mathcal V$-valued S-polynomial, and hence, by \eqref{Sf1} of Lemma \ref{SPVlem}, so is each $Q_n(t)$.
The uniqueness of the S-polynomials $Q_n$'s follows from Lemma \ref{unilem}.
\end{proof}

Our proof of Theorem \ref{vthm} uses the Gevrey norm technique.
We recall a convenient estimate in \cite[Lemma 2.1]{HM1} for the Gevrey norms of the the bi-linear form $B(\cdot,\cdot)$, which is a generalization of the original inequality in \cite[Lemma 2.1]{FT-Gevrey}, and also the Sobolev estimates in \cite{FHOZ1}.

\textit{
There exists a constant $K\ge 1$ such that for any numbers $\alpha\ge 1/2$, $\sigma\ge 0$, and any functions $v,w\in G_{\alpha+1/2,\sigma}$,  one has
\beq\label{Buv}
|B(v,w)|_{\alpha,\sigma}
\le K^\alpha |v|_{\alpha+1/2,\sigma}  |w|_{\alpha+1/2,\sigma}.
\eeq
}

The same estimate as \eqref{Buv} can be obtained for $B(t,v,w)$.

\begin{lemma}\label{Btuv}
For any numbers $\alpha\ge 1/2$, $\sigma\ge 0$, any functions $v,w\in G_{\alpha+1/2,\sigma}$  and any $t\in \R$, one has
\beq\label{Bas}
|B(t,v,w)|_{\alpha,\sigma}
\le K^\alpha |v|_{\alpha+1/2,\sigma}  |w|_{\alpha+1/2,\sigma},
\eeq
where $K$ is the constant in \eqref{Buv}.
\end{lemma}
\begin{proof}
 By the isometry \eqref{ESiso} and inequality \eqref{Buv}, we have
 \begin{align*}
  |B(t,v,w)|_{\alpha,\sigma}&=|B(e^{-tS}v,e^{-tS}w)|_{\alpha,\sigma}\\
&  \le K^\alpha |e^{-tS} v|_{\alpha+1/2,\sigma}  |e^{-tS} w|_{\alpha+1/2,\sigma}
  =  K^\alpha |v|_{\alpha+1/2,\sigma}  |w|_{\alpha+1/2,\sigma},
\end{align*}
which proves \eqref{Bas}.
\end{proof}

As another preparation for the proof of Theorem \ref{vthm}, we establish the relevant estimates for the Gevrey norms of $v(t)$, when $t$ is large.

\begin{proposition}
\label{decay2}
Let $v^0\in H$ and $v(t)$ be a Leray-Hopf weak solution of \eqref{veq}.
For any $\s>0$,   there exist $T,D_\s>0$  such that
	\beqs
|v(t)|_{1/2,\s+1}\le D_\s e^{-t},\quad \text{ for all } t\ge T.
	\eeqs
Moreover, for any $\alpha\ge 0$ there exists $D_{\al,\s}>0$ such that
\beq\label{cl10}
|v(t)|_{\al+1/2,\s}\le D_{\al,\s} e^{-t},\quad \text{ for all }\ t\ge T.
\eeq
\end{proposition}
\begin{proof}
Thanks to the isometry \eqref{ESiso}, properties \eqref{BS1}, \eqref{BS2}, and inequality \eqref{Bas}, the proof, with $\mu_1=1$ under the current setting, is exactly the same as in \cite[Theorem 2.4]{HM1} and is omitted.
\end{proof}

We now are ready to prove Theorem \ref{vthm}.

\begin{proof}[Proof of Theorem \ref{vthm}]
This proof follows \cite{FS87} and \cite{HM1,HM2} with necessary modifications.

Firstly, we note, by part (iii) of Theorem \ref{typicalthm},  that there exists $T_*>0$ such that
equation \eqref{veq} holds in the classical sense in $\mathcal D(A)$ on $(T_*,\infty)$.

Let $\sigma>0$ be fixed. For each $N\in\N$, our main statement is

\textbf{(H$_N$)}
\textit{There exist $\mathcal V$-valued S-polynomials $q_n$'s for $n=1,2,\ldots,N$,
such that
\beq\label{vnrem}
\Big|v(t)-\sum_{n=1}^N q_n(t)e^{-\mu_n t}\Big|_{\alpha,\sigma}=\mathcal O(e^{-(\mu_N+\varep)t})\text{ as }t\to\infty,
\eeq
for all $\alpha>0$, and some $\varep=\varep_{N,\alpha}>0$. Moreover, each $v_n(t)\eqdef q_n(t)e^{-\mu_n t}$, for $n=1,2,\ldots,N$, solves the equation
\beq\label{vneq}
v_n'(t)+Av_n(t) +\sum_{\stackrel{1\le m,j \le n-1}{\mu_m+\mu_j=\mu_n}} B_{\Omega}(t,v_m(t),v_j(t))=0, \text{ for all } t\in\R.
\eeq
}

\noindent\textbf{Claim:} \textbf{(H$_N$)} holds true for any $N\in\N$.

We prove this Claim by induction in $N$.

\medskip
\noindent\textbf{First step $N=1$.}
Given $\alpha\ge 1/2$.
By estimate \eqref{cl10}, there exist $T_0>T_*$ and $d_0>0$ such that
\beq\label{vdec}
|v(t)|_{\alpha+1/2,\sigma}\le d_0 e^{-\mu_1 t},\quad\text{for all }t\ge T_0.
\eeq

Set $w_0=e^{\mu_1 t}v(t)$. We have, for $t\in(T_0,\infty)$, that
\beqs
w_0'=e^{\mu_1 t}( v' +\mu_1 v)
=e^{\mu_1 t} (-Av -B_{\Omega}(t,v,v)) +\mu_1 v),
\eeqs
hence
\beq \label{w0eq}
w_0' +(A-\mu_1)w_0 =H_0(t)\eqdef - e^{\mu_1 t}B_{\Omega}(t,v(t),v(t)).
\eeq

Estimate \eqref{vdec} and inequality \eqref{Bas} imply
\beq\label{HOT}
|H_0(T_0+t)|_{\alpha,\sigma}
\le e^{\mu_1 (T_0+t)} K^\alpha |v(T_0+t)|_{\alpha+1/2,\sigma}^2
\le M_0 e^{-\mu_1 t},\quad\text{for all } t\ge 0,
\eeq
where $M_0=K^\alpha d_0^2 e^{-\mu_1 T_0}$.

For $k\in\N$, applying the projection $R_{\Lambda_k}$ to equation \eqref{w0eq} gives
\beq \label{Rkw0}
(R_{\Lambda_k} w_0)' +(\Lambda_k-\mu_1)R_{\Lambda_k} w_0 =R_{\Lambda_k} H_0(t).
\eeq

We apply Lemma \ref{polylem} to equation \eqref{Rkw0} in the space $X=R_{\Lambda_k}H$ with norm $\|\cdot\|_X=|\cdot|_{\alpha,\sigma}$,  solution $y(t)=R_{\Lambda_k} w_0(T_0+t)$, S-polynomial  $p(t)\equiv 0$, constant  $\beta=\Lambda_k-\mu_1\ge 0$, function $g(t)=R_{\Lambda_k}H_0(T_0+t)$ and numbers $M=M_0$, $\delta=\mu_1$ in \eqref{gM}.

When $k=1$, we have $\beta=0$, then by Lemma \ref{polylem}(ii), it follows that
\beq\label{R1w0}
|R_{\Lambda_1} w_0(T_0+t)-\xi_1|_{\alpha,\sigma}=\mathcal O(e^{-\mu_1 t}),
\eeq
where
\beqs
\xi_1=R_{\Lambda_1}w_0(T_0)+\int_0^\infty e^{\mu_1 \tau} R_{\Lambda_1} H_0(T_0+\tau)d\tau,
\eeqs
which exists and belongs to $R_{\Lambda_1}H$.

When $k\ge 2$, we have $\beta\ge \mu_2-\mu_1>0$, and it follows Lemma \ref{polylem}(i) that $q(t)$ defined by \eqref{qdef} is $0$, and, by \eqref{zsquare}, one has
\begin{align*}
&|R_{\Lambda_k} w_0(T_0+t)|_{\alpha,\sigma}^2
\le 2e^{-2(\Lambda_k-\mu_1) t}|R_{\Lambda_k} w_0(T_0)|_{\alpha,\sigma}^2 + 2t \int_0^t e^{-2(\Lambda_k-\mu_1)(t-\tau)} |R_{\Lambda_k} H_0(T_0+\tau)|_{\alpha,\sigma}^2 d\tau  \\
&\le 2e^{-2(\mu_2-\mu_1) t}|R_{\Lambda_k} w_0(T_0)|_{\alpha,\sigma}^2 + 2t \int_0^t e^{-2(\mu_2-\mu_1)(t-\tau)} |R_{\Lambda_k} H_0(T_0+\tau)|_{\alpha,\sigma}^2 d\tau.
\end{align*}

Summing up this inequality in $k$ gives
\begin{align*}
&|({\rm Id}-R_{\Lambda_1})w_0(T_0+t)|_{\alpha,\sigma}^2
=\sum_{k=2}^\infty |R_{\Lambda_k} w_0(T_0+t)|_{\alpha,\sigma}^2\\
&\le 2e^{-2(\mu_2-\mu_1) t}\sum_{k=2}^\infty |R_{\Lambda_k} w_0(T_0)|_{\alpha,\sigma}^2 + 2t \int_0^t e^{-2(\mu_2-\mu_1)(t-\tau)}\sum_{k=2}^\infty |R_{\Lambda_k} H_0(T_0+\tau)|_{\alpha,\sigma}^2 d\tau\\
&\le 2e^{-2(\mu_2-\mu_1) t}|w_0(T_0)|_{\alpha,\sigma}^2 + 2te^{-2(\mu_2-\mu_1)t} \int_0^t e^{2(\mu_2-\mu_1)\tau} | H_0(T_0+\tau)|_{\alpha,\sigma}^2 d\tau.
\end{align*}

Using \eqref{HOT}, we obtain
\begin{align*}
|({\rm Id}-R_{\Lambda_1})w_0(T_0+t)|_{\alpha,\sigma}^2
\le 2e^{-2(\mu_2-\mu_1) t}\Big(|w_0(T_0)|_{\alpha,\sigma}^2 +  2 M_0 t \int_0^t e^{2(\mu_2-2\mu_1)\tau} d\tau\Big).
\end{align*}

Since $2\mu_1>\mu_1$ and, by  \eqref{musum}, $2\mu_1\in \langle \sigma(A)\rangle$, then, thanks to \eqref{mudef}, $2\mu_1\ge \mu_2$.
Using this simple fact to estimate the last integral yields
\begin{align}\label{IR1w0}
|({\rm Id}-R_{\Lambda_1})w_0(T_0+t)|_{\alpha,\sigma}^2
&\le 2e^{-2(\mu_2-\mu_1) t}\big(|w_0(T_0)|_{\alpha,\sigma}^2 +  2 M_0 t^2\big).
\end{align}

Combining \eqref{R1w0}, \eqref{IR1w0} and the fact $\mu_2-\mu_1\le \mu_1$, gives
\beq\label{w0x1}
|w_0(t)-\xi_1|_{\alpha,\sigma}
\le |R_{\Lambda_1}w_0(t)-\xi_1|_{\alpha,\sigma}  + |({\rm Id}-R_{\Lambda_1})w_0(t)|_{\alpha,\sigma}=\mathcal O(e^{-\varep t}),
\eeq
for any number $\varep$ such that $0<\varep<\mu_2-\mu_1$.

Define
\beq\label{q1} q_1(t)\equiv \xi_1.\eeq

Multiplying \eqref{w0x1} by $e^{-\mu_1 t}$ yields \eqref{vnrem} for $N=1$.
Also, since $\xi_1\in R_{\Lambda_1}H$, it is clear that $v_1(t)=\xi_1 e^{-\Lambda_1 t}$ satisfies  $v_1'(t)+Av_1(t)=0$ on $\R$.
Hence, $v_1$ satisfies \eqref{vneq} with $n=1$, because the sum of the bi-linear terms in \eqref{vneq} is void.
Note that $q_1$ does not depend on $\alpha$. Therefore, the statement \textbf{(H$_N$)} holds true for $N=1$.

\medskip
\noindent\textbf{Induction step.} Let $N\ge 1$ and assume the statement \textbf{(H$_N$)} holds true.
Let $q_n$, for $n=1,2,\ldots, N,$ be the $\mathcal V$-valued  S-polynomials in \textbf{(H$_N$)}.
There exists $\Lambda\in\sigma(A)$ such that
\beq\label{qnL}
q_n\in\SP_1(P_\Lambda H), \text{ for all } 1\le n\le N.
\eeq

Let $v_n(t)=q_n(t)e^{-\mu_n t}$, denote $s_N(t)=\sum_{n=1}^N v_n(t)$
and $r_N(t)=v(t)-s_N(t)$. Let $\alpha\ge 1/2$.

(a) By the definition of $v_n$, we have, for $n\ge 2$,
\beq\label{h3}
|v_n(t)|_{\alpha+1/2,\sigma}= \mathcal O(e^{-(\mu_n-\delta) t})\quad\forall \delta>0,
\eeq
and, thanks to \eqref{q1},
\beqs
|v_1(t)|_{\alpha+1/2,\sigma}= \mathcal O(e^{-\mu_1 t}).
\eeqs
The last two properties imply
\beq\label{h2}
|s_N(T+t)|_{\alpha+1/2,\sigma} =\mathcal O(e^{-\mu_1 t}).
\eeq

By the induction hypothesis \textbf{(H$_N$)} applied to $\alpha+1/2$, there exists $\varep>0$ such that
\beq\label{h1}
|r_N(t)|_{\alpha+1/2,\sigma}=\mathcal O(e^{-(\mu_N+\varep)t}).
\eeq

We derive a differential equation for $r_N(t)$, for $t>T_*$. We calculate from \eqref{veq} and \eqref{vneq} for $n=1,2,\ldots,N$ that
\begin{align*}
 r_N'
 &=v'-\sum_{n=1}^N v_n' =-Av-B_\Omega(t,v,v)-\sum_{n=1}^N \Big\{ -Av_n - \sum_{\stackrel{1\le m,j \le n-1}{\mu_m+\mu_j=\mu_n}} B_{\Omega}(t,v_m,v_j)\Big\}\\
 &=-Ar_N -B_\Omega(t,r_N(t),v(t))-B_\Omega(t,s_N(t),r_N(t))-B_\Omega(t,s_N(t),s_N(t))\\
 &\quad + \sum_{\mu_m+\mu_j\le \mu_N} B_{\Omega}(t,v_m,v_j).
\end{align*}

We manipulate the last two terms as
\begin{align*}
&-B_\Omega(t,s_N(t),s_N(t))+ \sum_{\mu_m+\mu_j\le \mu_N} B_{\Omega}(t,v_m,v_j)
=-\sum_{\stackrel{1\le m,j\le N}{\mu_m+\mu_j\ge\mu_{N+1}}} B_\Omega(t,v_m(t),v_j(t))\\
&  =-\sum_{\stackrel{1\le m,j\le N}{\mu_m+\mu_j=\mu_{N+1}}} B_\Omega(t,v_m,v_j)-\sum_{\stackrel{1\le m,j\le N}{\mu_m+\mu_j\ge\mu_{N+2}}} B_\Omega(t,v_m(t),v_j(t)).
\end{align*}

 Thus, we obtain
\beq\label{vNeq}
r_N'(t)+Ar_N(t) +\sum_{\stackrel{1\le m,j\le N}{\mu_m+\mu_j=\mu_{N+1}}} B_\Omega(t,v_m(t),v_j(t))=h_N,
\text{ for } t>T_*,
\eeq
where
\beq\label{hdef}
h_N(t)=-B_\Omega(t,r_N(t),v(t))-B_\Omega(t,s_N(t),r_N(t))-\sum_{\stackrel{1\le m,j\le N}{\mu_m+\mu_j\ge  \mu_{N+2}}} B_\Omega(t,v_m(t),v_j(t)).
\eeq

(b) We estimate each term on the right-hand side of \eqref{hdef}. Note from \eqref{vdec}, \eqref{h1}, \eqref{h3}, \eqref{h2} and \eqref{Bas} that
\begin{align}
\label{BO1}
  |B_\Omega(t,r_N(t),v(t))|_{\alpha,\sigma} &=\mathcal O(e^{-(\mu_N+\mu_1+\varep)t}),  \\
\label{BO2}
  |B_\Omega(t,s_N(t),r_N(t))|_{\alpha,\sigma} &= \mathcal O(e^{-(\mu_N+\mu_1+\varep)t}),
\end{align}
and for $1\le m,j\le N$ with $\mu_m+\mu_j\ge  \mu_{N+2}$,
\beq
|B_\Omega(t,v_m(t),v_j(t))|_{\alpha,\sigma}=\mathcal O(e^{-(\mu_m+\mu_j-2\delta)t})=\mathcal O(e^{-(\mu_{N+2}-2\delta)t}),\forall \delta>0.\label{Bsum}
\eeq

Observe that $\mu_N+\mu_1>\mu_N$, and, by \eqref{musum}, $\mu_N+\mu_1\in\langle \sigma(A)\rangle$. This and \eqref{mudef} imply $\mu_N+\mu_1\ge \mu_{N+1}$. Obviously, $\mu_{N+2}>\mu_{N+1}$.
Then, by taking $\delta$ sufficiently small in \eqref{Bsum}, we have from \eqref{hdef}, \eqref{BO1}, \eqref{BO2} and \eqref{Bsum} that
\beq\label{hNO}
|h_N(t)|_{\alpha,\sigma}=\mathcal O(e^{-(\mu_{N+1}+\delta_N)t}),\quad  \text{for some }\delta_N\in(0,\mu_{N+2}-\mu_{N+1}).
\eeq

(c) Define $w_{N}(t)=e^{\mu_{N+1}t}r_N(t)$, and  $w_{N,k}(t)=R_{\Lambda_k}w_{N}(t)$, for $k\in\N$.
We have from \eqref{vNeq} that
\beq\label{wNk}
\ddt w_{N,k} +(\Lambda_k-\mu_{N+1})w_{N,k}=-\sum_{\stackrel{1\le m,j\le N}{\mu_m+\mu_j=\mu_{N+1}}} R_{\Lambda_k}B_\Omega(t,q_m,q_j) +R_{\Lambda_k}H_N(t),
\eeq
where $H_N(t)=e^{\mu_{N+1}t}h_N(t)$.
By \eqref{qnL} and \eqref{Sf3},
\beq \label{BLam}
B_\Omega(t,q_m(t),q_j(t))\in \SP_1(P_{4\Lambda}H),
\eeq
which implies that the finite sum
$$\sum_{\stackrel{1\le m,j\le N}{\mu_m+\mu_j=\mu_{N+1}}} B_\Omega(t,q_m(t),q_j(t))\in \SP_1(P_{4\Lambda}H).$$
Consequently,
\beq\label{B4L}
\sum_{\stackrel{1\le m,j\le N}{\mu_m+\mu_j=\mu_{N+1}}} R_{\Lambda_k}B_\Omega(t,q_m(t),q_j(t))
\in \SP_1(R_{\Lambda_k}H).
\eeq

By the first property in Lemma \ref{Spolyinvar}(i),
\beqs
\sum_{\stackrel{1\le m,j\le N}{\mu_m+\mu_j=\mu_{N+1}}} R_{\Lambda_k}B_\Omega(T+t,q_m(T+t),q_j(T+t))\in \SP_1(R_{\Lambda_k}H), \quad \text{for all } T\in \R.
\eeqs

By \eqref{hNO},
$|H_N(t)|_{\alpha,\sigma}=\mathcal O(e^{-\delta_N t})$.
Then there exist $T_N>T_*$ and $M_N>0$ such that
\beq\label{HNrate}
|H_N(T_N+t)|_{\alpha,\sigma}\le M_N e^{-\delta_N t}, \quad\text{for all }t\ge 0.
\eeq

We will apply Lemma \ref{polylem} again to equation \eqref{wNk} in the space $X=R_{\Lambda_k}H$ with norm $\|\cdot\|_X=|\cdot|_{\alpha,\sigma}$, solution $y(t)=w_{N,k}(T_N+t)$, constant $\beta=\Lambda_k-\mu_{N+1}$, S-polynomial
$$p(t)=-\sum_{\stackrel{1\le m,j\le N}{\mu_m+\mu_j=\mu_{N+1}}} R_{\Lambda_k}B_\Omega(T_N+t,q_m(T_N+t),q_j(T_N+t)),$$
function $g(t)=R_{\Lambda_k}H_N(T_N+t)$, numbers $M=M_N$ and $\delta=\delta_N$ in \eqref{HNrate}.

We consider three cases.

\medskip
\textbf{Case $\Lambda_k= \mu_{N+1}$.}  Then $\beta=0$ in \eqref{wNk}.
Let
\beqs
\xi_{N+1}\eqdef R_{\mu_{N+1}}r_N(T_N)+\int_0^\infty e^{\mu_{N+1}\tau} R_{\mu_{N+1}}H_N(T_N+\tau)\ d\tau,
\eeqs
which exists and belongs to $R_{\mu_{N+1}}H$.
Define
\beq \label{p1}
p_{N+1,k}(t)=\xi_{N+1}-\int_0^t \sum_{\stackrel{1\le m,j\le N}{\mu_m+\mu_j=\mu_{N+1}}} R_{\mu_{N+1}} B_\Omega(T_N+ \tau,q_m(T_N+\tau),q_j(T_N+\tau))d\tau.
\eeq

\medskip
\textbf{Case $\Lambda_k\le \mu_N$.} Then $\beta<0$ in \eqref{wNk}.
Note, by \eqref{h1}, that
\beqs
e^{\beta t}|y(t)|_{\alpha,\sigma}=e^{\Lambda_k t}|R_{\Lambda_k}r_N(T_N+t)|_{\alpha,\sigma}
\le e^{\mu_N t}|R_{\Lambda_k}r_N(T_N+t)|_{\alpha,\sigma}=\mathcal O(e^{-\varep t}).
\eeqs
Hence condition \eqref{limycond} is met. Define
\beq\label{i1}
\begin{aligned}
p_{N+1,k}(t)&=e^{-(\Lambda_k-\mu_{N+1})t}\int_t^\infty e^{(\Lambda_k-\mu_{N+1})\tau}\\
&\quad \cdot \Big(\sum_{\stackrel{1\le m,j\le N}{\mu_m+\mu_j=\mu_{N+1}}}R_{\Lambda_k}B_\Omega(T_N+ \tau,q_m(T_N+\tau),q_j(T_N+\tau)) \Big) d\tau.
\end{aligned}
\eeq

In the above two cases of $\Lambda_k$,  by applying Lemma \ref{polylem}(ii), we obtain $p_{N+1,k}\in \SP_1(R_{\Lambda_k}H)$
that, by \eqref{g1b3}, satisfies
\beq\label{smallk}
|w_{N,k}(T_N+t)-p_{N+1,k}(t)|_{\alpha,\sigma}^2=\mathcal O(e^{-2\delta_N t}).
\eeq

\medskip
\textbf{Case $\Lambda_k\ge \mu_{N+2}$.} Then $\beta>0$ in \eqref{wNk}. Define
\beq\label{i2}
\begin{aligned}
p_{N+1,k}(t)
&=-e^{-(\Lambda_k-\mu_{N+1})t}\int_{-\infty}^t e^{(\Lambda_k-\mu_{N+1})\tau} \\
&\quad \cdot \Big(\sum_{\stackrel{1\le m,j\le N}{\mu_m+\mu_j=\mu_{N+1}}}R_{\Lambda_k}B_\Omega(T_N+ \tau,q_m(T_N+\tau),q_j(T_N+\tau)) \Big) d\tau.
\end{aligned}
\eeq

Applying Lemma \ref{polylem}(i), we obtain $p_{N+1,k}\in \SP_1(R_{\Lambda_k}H)$, and, by \eqref{zsquare},
\beq  \label{tilwNk}
\begin{aligned}
|w_{N,k}(T_N+t)-p_{N+1,k}(t)|_{\alpha,\sigma}^2
&\le 2 e^{-2(\Lambda_k-\mu_{N+1})t} |w_{N,k}(T_N)- p_{N+1,k}(0)|_{\alpha,\sigma}^2 \\
&\quad + 2t\int_0^t e^{-2(\mu_{N+2}-\mu_{N+1})(t-\tau)} |R_{\Lambda_k} H_N(T_N+\tau)|_{\alpha,\sigma}^2 d\tau.
\end{aligned}
\eeq

\medskip
Denote $z_{N,k}(t)= w_{N,k}(T_N+t)-p_{N+1,k}(t)$, which is the remainder on the left-hand side of \eqref{smallk} and \eqref{tilwNk}.

On the one hand, because there are only finitely many $k$'s with $\Lambda_k\le \mu_{N+1}$, it follows from \eqref{smallk} that
\beq\label{sumsmall}
\sum_{\Lambda_k\le \mu_{N+1}} |z_{N,k}(t)|_{\alpha,\sigma}^2=\mathcal O(e^{-2\delta_N t}).
\eeq

On the other hand, summing inequality \eqref{tilwNk} in $k$ for which $\Lambda_k\ge \mu_{N+2}$, we obtain
\begin{align}
\sum_{\Lambda_k\ge \mu_{N+2}} |z_{N,k}(t)|_{\alpha,\sigma}^2
&\le  4 e^{-2(\mu_{N+2}-\mu_{N+1})t} \sum_{\Lambda_k\ge \mu_{N+2}}  |w_{N,k}(T_N)- p_{N+1,k}(0)|_{\alpha,\sigma}^2 \notag \\
&\quad + 2t\int_0^t e^{-2(\mu_{N+2}-\mu_{N+1})(t-\tau)}\sum_{\Lambda_k\ge \mu_{N+2}} |R_{\Lambda_k} H_N(T_N+\tau)|_{\alpha,\sigma}^2 d\tau.\label{zz}
\end{align}

For the terms on the right-hand side of \eqref{zz}, we have
\begin{align*}
\sum_{\Lambda_k\ge \mu_{N+2}}  |w_{N,k}(T_N)- p_{N+1,k}(0)|_{\alpha,\sigma}^2
&\le 2\sum_{\Lambda_k\ge \mu_{N+2}}  |R_{\Lambda_k} w_N(T_N)|_{\alpha,\sigma}^2+2\sum_{\Lambda_k\ge \mu_{N+2}}|p_{N+1,k}(0)|_{\alpha,\sigma}^2\\
&\le 2|w_N(T_N)|_{\alpha,\sigma}^2+2\sum_{\Lambda_k\ge \mu_{N+2}}|p_{N+1,k}(0)|_{\alpha,\sigma}^2.
\end{align*}
It follows from \eqref{B4L} and \eqref{i2} that $p_{N+1,k}\equiv 0$, for $\Lambda_k>\max\{4\Lambda,\mu_{N+1}\}$.
Thus, $$\sum_{\Lambda_k\ge \mu_{N+2}}|p_{N+1,k}(0)|_{\alpha,\sigma}^2\text{ is a finite sum, and hence is finite.}
$$
Also, the last integral in \eqref{zz} has
\beqs
\sum_{\Lambda_k\ge \mu_{N+2}} |R_{\Lambda_k} H_N(T_N+\tau)|_{\alpha,\sigma}^2\le |H_N(T_N+\tau)|_{\alpha,\sigma}^2\le M_N^2 e^{-2\delta_N \tau}.
\eeqs

Therefore, there exists $C_0>0$ such that
\begin{align*}
\sum_{\Lambda_k\ge \mu_{N+2}} |z_{N,k}(t)|_{\alpha,\sigma}^2
&\le  4C_0 e^{-2(\mu_{N+2}-\mu_{N+1})t} + 2M_N^2t e^{-2(\mu_{N+2}-\mu_{N+1})t} \int_0^t e^{2(\mu_{N+2}-\mu_{N+1}-\delta_N)\tau} d\tau.
\end{align*}

Recall that $\delta_N<\mu_{N+2}-\mu_{N+1}$ in \eqref{HNrate}. Calculating the last integral explicitly, we easily find
\beq\label{sumlarge}
\sum_{\Lambda_k\ge \mu_{N+2}} |z_{N,k}(t)|_{\alpha,\sigma}^2
\le  4C_0 e^{-2(\mu_{N+2}-\mu_{N+1})t} +  \frac{M_N^2 t  e^{-2\delta_N t}}{\mu_{N+2}-\mu_{N+1}-\delta_N}\le C_N e^{-\delta_N t},
\eeq
for some constant $C_N>0$.

Combining \eqref{sumsmall} and \eqref{sumlarge} gives
\beq\label{sumall}
\sum_{k=1}^\infty |z_{N,k}(t)|_{\alpha,\sigma}^2=\mathcal O(e^{-\delta_N t}).
\eeq

(d) By the first limit in \eqref{speclim}, we can select a number $\Lambda_*\in\sigma(A)$ to be the smallest $\Lambda_n$ such that $\Lambda_n\ge\max\{4\Lambda,\mu_{N+1}\}$.
By \eqref{BLam} and formulas \eqref{p1}, \eqref{i1}, \eqref{i2}, we have
\beq \label{pprop}
p_{N+1,k}\in \SP_1(P_{\Lambda_k}H)\subset \SP_1(P_{\Lambda_*}H),\text{ for }\Lambda_k\le \Lambda_*,
 \text{ and }
p_{N+1,k}=0, \text{ for }\Lambda_k>\Lambda_*.
\eeq

Define, for $t\in\R$,
\beq\label{qq}
q_{N+1}(t)=\sum_{k=1}^\infty p_{N+1,k}(t-T_N).
\eeq
Thanks to \eqref{pprop}, the sum in \eqref{qq} is only a finite sum, and
$q_{N+1}\in \SP_1(P_{\Lambda_*}H)$.

Obviously,
\beq \label{Rqp}
R_{\Lambda_k}q_{N+1}(t+T_N)=p_{N+1,k}(t),
\eeq
hence, $z_{N,k}(t)=R_{\Lambda_k}\big(w_N(t+T_N)-q_{N+1}(t+T_N)\big)$.
It follows from \eqref{sumall} that
\beqs
|w_{N}(t+T_N)-q_{N+1}(t+T_N)|_{\alpha,\sigma}^2=\sum_{k=1}^\infty |z_{N,k}(t)|_{\alpha,\sigma}^2
=\mathcal O(e^{-\delta_N t}),
\eeqs
which implies that
\beqs
|w_{N}(t)-q_{N+1}(t)|_{\alpha,\sigma}=\mathcal O(e^{-\delta_N t/2}).
\eeqs
Multiplying this equation by $e^{-\mu_{N+1}t}$, we obtain the desired statement \eqref{vnrem} for $N+1$.

(e) It remains to prove the ODE \eqref{vneq}, for $n=N+1$. By \eqref{pode2} of Lemma \ref{polylem}, we have, for each $k$, that
$$p_{N+1,k}'(t)+(\Lambda_k-\mu_{N+1})p_{N+1,k}(t)+ \sum_{\stackrel{1\le m,j\le N}{\mu_m+\mu_j=\mu_{N+1}}} R_{\Lambda_k}B_{\Omega}(T_N+t,q_m(T_N+t),q_j(T_N+t)) =0.$$

From this and \eqref{Rqp}, we deduce
$$(R_{\Lambda_k} q_{N+1}(t) )'+(\Lambda_k-\mu_{N+1})R_{\Lambda_k}q_{N+1}(t) + R_{\Lambda_k}\sum_{\stackrel{1\le m,j\le N}{\mu_m+\mu_j=\mu_{N+1}}} B_{\Omega}(t,q_m(t),q_j(t)) =0.$$

Multiplying this equation by $e^{-\mu_{N+1}t}$ yields
$$(R_{\Lambda_k} v_{N+1}(t))'+A(R_{\Lambda_k}v_{N+1}(t)) + R_{\Lambda_k}\sum_{\stackrel{1\le m,j\le N}{\mu_m+\mu_j=\mu_{N+1}}} B_{\Omega}(t,v_m(t),v_j(t)) =0,$$
where $v_{N+1}(t)= q_{N+1}(t)e^{-\mu_{N+1}t}$.
Note in this equation that $v_{N+1}(t)$ and  $B_{\Omega}(t,v_m(t),v_j(t))$ all belong to $P_{\Lambda_*}H$.
Then summing up the equation in $k$, for which $\Lambda_k\le \Lambda_*$, yields that  the ODE system \eqref{vneq} holds for $n=N+1$.
According to Remark \ref{uniG}, the S-polynomial $q_{N+1}$ is independent of $\alpha$.
Therefore, the statement \textbf{(H$_{N+1}$)} is proved.

By the induction principle, the statement \textbf{(H$_N$)} holds true for all $N$, i.e., the Claim is proved.

Now, we note that at step $n=N+1$, the $q_n$'s, for $n=1,2,\ldots,N$, are those from the step $n=N$, and are used in the construction of $q_{N+1}$. Therefore, for each $\sigma>0$, such recursive construction gives the existence of  the S-polynomials $q_n$'s, for all $n\in\N$. By Remark \ref{uniG} again, all these $q_n$'s are, in fact,  independent of $\sigma$. Therefore, \eqref{vnrem} holds true for all $N$, $\alpha$, $\sigma$. Then estimate \eqref{vremain} follows thanks to \eqref{ex3}.
Finally, by Lemma \ref{unilem}, the S-polynomials $q_n$'s are unique.
The proof is complete.
\end{proof}

\begin{remark}
 The statement and proof of Theorem \ref{vthm} can be presented simply with $G_{0,\sigma}$ for all $\sigma\ge 0$. Nonetheless, general calculations in the above proof for $G_{\alpha,\sigma}$ are flexible and can be applied to cover the case when a non-potential force is included in the NSE and has limited regularity in $G_{\alpha,\sigma}$ for a fixed $\sigma$, see \cite{HM2,CaH1,CaH2}.
\end{remark}

\begin{remark}\label{limrmk}
 It is not known whether the polynomials $Q_n$'s in Theorem \ref{uthm} have any limits as $|\Omega|\to\infty$. However, such a question can be answered for some special solutions,
 see Remark \ref{lim2rmk} below.
\end{remark}

\section{The case of non-zero spatial average solutions}
\label{nonzerosec}

In this section, we establish the asymptotic expansions for the solutions with non-zero spatial averages.
Below, $C_{\vecx,t}^{2,1}$ denotes the class of continuous functions $f(\vecx,t)$ with continuous partial derivatives $\partial f/\partial x_k,\partial^2 f/\partial x_k\partial x_j$ (for $k,j=1,2,3$,) and $\partial f/\partial t$.
Notation $C_{\vecx}^{1}$ is meant similarly.

\begin{assumption}
 Throughout this section, $\vecu(\vecx,t)\in C_{\vecx,t}^{2,1}(\R^3\times (0,\infty))\cap C(\R^3\times[0,\infty))$ and $p(\vecx,t)\in C_{\vecx}^{1}(\R^3\times (0,\infty))$
 are $\vecL$-periodic  functions that form a solution $(\vecu,p)$ of the NSE \eqref{NPDE} and \eqref{divfree}.
\end{assumption}

\medskip
Note that any $\vecL$-periodic function on $\R^3$ can be considered as a function on the flat torus $\domL$.
In particular, $u(t)=\vecu(\cdot,t)$ and $p(t)=p(\cdot,t)$ can be considered as functions on $\domL$.

Suppose $\vecf$ is a $\vecL$-periodic vector field on $\R^3$ and $\vecf\in \mathcal D(A^\alpha e^{\sigma A^{1/2}})$, for some $\alpha,\sigma\ge 0$.
Let $\vecg(\vecx)=\vecf(\vecx+\vecX_0)$, for some fixed $\vecX_0\in\R^3$, then $\vecg\in H$.
Let $\vecf(\vecx)=\sump  \widehat\vecf_\veck e^{i \widecheck \veck\cdot \vecx}$, then
\beqs
\vecg(\vecx)=\sump  \widehat\vecg_\veck e^{i \widecheck \veck\cdot \vecx}, \text{ where } \widehat\vecg_\veck=\widehat\vecf_\veck e^{i\widecheck\veck\cdot \vecX_0}.
\eeqs
It follows that
$|\widehat\vecg_\veck|=|\widehat\vecf_\veck|$, and consequently,
 \beq\label{normshift}
| \vecf(\cdot+\vecX_0) |_{\alpha,\sigma}=| \vecg(\cdot) |_{\alpha,\sigma}=| \vecf(\cdot)|_{\alpha,\sigma}.
 \eeq

\medskip
Returning to the solution $(\vecu,p)$, denote, for $t\ge 0$,
\beqs
\vecU(t)=\frac1{L_1L_2L_3}\int_\domL \vecu(\vecx,t)d\vecx.
\eeqs

Integrating equation \eqref{NPDE} over the domain $\domL$ gives
\beqs
\vecU'(t)+\Omega J \vecU(t)=0,\quad t>0.
\eeqs
Hence,
\beq \label{Ut}
\vecU(t)=e^{-\Omega t J} \vecU(0)=\begin{pmatrix}
                           \cos(\Omega t)& \sin(\Omega t)&0\\
                           -\sin(\Omega t)&\cos(\Omega t)&0\\
                           0&0&1
                          \end{pmatrix} \vecU(0).
\eeq

\begin{theorem}\label{SSthm}
There exist $\mathcal V$-valued SS-polynomials $\mathcal Q_n(t)$'s, for all $n\in \N$, such that
 \beq\label{uU}
 u(t)-\vecU(t)\sim \sum_{n=1}^\infty \mathcal Q_n(t) e^{-\mu_n t} \text{ in } G_{\alpha,\sigma},\text { for all }\alpha,\sigma>0.
 \eeq
\end{theorem}
\begin{proof}
For $t\ge 0$, define
$\vecV(t)=\int_0^t \vecU(\tau)d\tau$,
which, by \eqref{Ut}, is
\beq \label{Vt}
\vecV(t)= \frac1\Omega\begin{pmatrix}
                           \sin(\Omega t)& 1-\cos(\Omega t)&0\\
                           \cos(\Omega t)-1&\sin(\Omega t)&0\\
                           0&0& \Omega t
                          \end{pmatrix} \vecU(0).
\eeq

We use the following hyper-Galilean transformation \cite{Foias-TAMU2002}:
\beq\label{Gtrans}
\vecw(\vecx,t)=\vecu(\vecx+\vecV(t),t)-\vecU(t), \quad \vartheta (\vecx,t)= p(\vecx+\vecV(t),t).
\eeq

By simple calculations, one can verify
\beqs
\vecw_t-\Delta \vecw +(\vecw\cdot \nabla) \vecw+\Omega J\vecw=-\nabla \vartheta
\text{ and } {\rm div }\, \vecw=0,
\eeqs
that is, $(\vecw,\vartheta)$ is also a classical solution of the system \eqref{NPDE} and \eqref{divfree} on $\R^3\times(0,\infty)$.

Note that  $(\vecw,\vartheta)$ is $\vecL$-periodic, and  $\vecw(\cdot,t)$ has zero average for each $t\ge 0$.
Applying Theorem \ref{uthm} to the solution $w(t)=\vecw(\cdot,t)$ we obtain
 \beq\label{wexp}
 w(t)\sim \sum_{n=1}^\infty Q_n(t) e^{-\mu_n t} \text{ in } G_{\alpha,\sigma },\text{ for all }\alpha,\sigma> 0,
 \eeq
 where  $Q_n(t)$'s are $\mathcal V$-valued S-polynomials.

Assume each $Q_n(t)$ is a mapping $\vecx\mapsto \vecQ_n(\vecx,t)$.
Note that each $Q_n(t)$ belongs to $\mathcal V$, hence $\vecQ_n(\vecx,t)$, as a function of $\vecx\in\R^3$, is $\vecL$-periodic.
Let $N\in\N$, $\alpha,\sigma>0$ and $\mu\in (\mu_N,\mu_{N+1})$. We have from \eqref{wexp} and \eqref{ex3} that
 \beqs
\Big| \vecu(\vecx+\vecV(t),t)-\vecU(t)-\sum_{n=1}^N \vecQ_n(\vecx,t)e^{-\mu_n t} \Big|_{\alpha,\sigma}=\mathcal O(e^{-\mu t}).
 \eeqs
This and \eqref{normshift} imply
  \beqs
\Big| \vecu(\vecx,t)-\vecU(t)-\sum_{n=1}^N \vecQ_n(\vecx-\vecV(t),t)e^{-\mu_n t} \Big|_{\alpha,\sigma}=\mathcal O(e^{-\mu t}),
 \eeqs
which means
 \beq\label{uUQ}
\Big| u(t)-\vecU(t)-\sum_{n=1}^N \mathcal Q_n(t) e^{-\mu_n t} \Big|_{\alpha,\sigma}=\mathcal O(e^{-\mu t}),
 \eeq
where  $\mathcal Q_n(t)=\vecQ_n(\cdot-\vecV(t),t)$ for $n\in\N$.
It remains to prove that $\mathcal Q_n\in \SP_2(\mathcal V)$, for all $n\in \N$.
Suppose
$$ \vecQ_n(\vecx,t)=\sump_{\rm{finitely\ many\ }\veck} \widehat \vecQ_{n,\veck}(t) e^{i\widecheck\veck\cdot\vecx},\text{ with }\widehat \vecQ_{n,\veck}(t)\in \SP_1(X_\veck).$$
 Then
 \beqs
\mathcal Q_n(t)
=\sump_{\rm{finitely\ many\ }\veck} \widehat \vecQ_{n,\veck}(t)e^{-i\widecheck\veck\cdot \vecV(t)}e^{i \widecheck\veck\cdot\vecx}
=\sump_{\rm{finitely\ many\ }\veck} \widehat {\mathcal Q}_{n,\veck}(t) e^{i\widecheck\veck\cdot\vecx},
 \eeqs
 where
 $\widehat {\mathcal Q}_{n,\veck}(t)=\widehat \vecQ_{n,\veck}(t)\big[\cos(\widecheck\veck\cdot \vecV(t))-i \sin(\widecheck\veck\cdot \vecV(t))\big]$.

Using the formula of $\vecV(t)$ in \eqref{Vt}, one can see that
\beqs
\widecheck\veck\cdot \vecV(t)=r_1\cos(\Omega t)+r_2\sin(\Omega t)+r_3 t + r_4,
\eeqs
for some numbers $r_2,r_2,r_3,r_4\in \R$, that depend on $\widecheck\veck$.

By using the product to sum formulas between trigonometric functions, we have
$$
\widehat \vecQ_{n,\veck}(t) \cos(\widecheck\veck\cdot \vecV(t)), \widehat \vecQ_{n,\veck}(t)  \sin(\widecheck\veck\cdot \vecV(t))
\in \SP_2(X_\veck).
$$
Therefore, $\widehat {\mathcal Q}_{n,\veck}(t)\in \SP_2(X_\veck)$, and, by Lemma \ref{SPlem1}, $\mathcal Q_n(t)\in \SP_2(\mathcal V)$.
With this fact, we obtain the expansion \eqref{uU} from estimate \eqref{uUQ}.
\end{proof}

\begin{remark}
We can roughly rewrite \eqref{uU} as an asymptotic expansion of $u(t)$ as
 \beqs
 u(t)\sim  \vecU(t) + \sum_{n=1}^\infty \mathcal Q_n(t) e^{-\mu_n t} .
 \eeqs
This shows that $\vecU(t)$ is the leading order term in the asymptotic approximation, as $t\to\infty$,  i.e., the  $\mathcal O(1)$ term, is explicitly determined by \eqref{Ut}, and is responsible for the non-zero average of $u(t)$. The following order terms in the approximation are all exponentially decaying and are with zero spatial averages.
\end{remark}

\section{Some special solutions}\label{specialsec}

We present next some special solutions of \eqref{NPDE} and \eqref{divfree}. They are inspired by examples in \cite{FS84b,FHN1,FHN2} for the case without rotation.

Let $\veck=(k_1,k_2,k_3)\in \Z^3\setminus\{\mathbf 0\}$ with $\widetilde\veck=(\widetilde k_1,\widetilde k_2,\widetilde k_3)$ be fixed.
For any $m\in\Z_*=\Z\setminus\{0\}$, we note from \eqref{kcheck}, \eqref{ktil} and \eqref{Jk} that
\beq \label{mkprop}
(m\veck)\widecheck{}=m\widecheck\veck,\quad
\widetilde{m\veck}={\rm sgn}(m)\widetilde\veck,\text{ and }
J_{m\veck}=\rm{sgn }(m)J_\veck.
\eeq

We define $V_\veck$ to be the space of all $u\in V$ such that
\beq \label{uset}
u=\sum _{m\in\Z_*} \widehat \vecu_{m\veck} e^{i m \widecheck\veck\cdot\vecx}.
\eeq

By \eqref{unabv},  one immediately sees that
\beq\label{specialB}
(u\cdot \nabla)v=0, \text{ for all } u,v\in V_\veck.
\eeq

Let $u\in V_\veck$ as in \eqref{uset} and $t\in\R$.  We have, by \eqref{etS}, \eqref{uset} and \eqref{mkprop},
\beqs
e^{tS}u=\sum_{m\in\Z_*} E_\veck(\widetilde k_3 t)\widehat \vecu_{m\veck} e^{i m \widecheck\veck\cdot\vecx}.
\eeqs
Hence
\beq\label{euV}
e^{tS}u\in V_\veck.
\eeq

Combining \eqref{specialB} and \eqref{euV} yields
\beq\label{specialBS}
[(e^{tS}u)\cdot\nabla ](e^{tS}v)=0,\text{ for all }u,v\in V_\veck \text{ and } t\in \R.
\eeq

\begin{theorem}\label{specthm1}
Let $\veck\in \Z^3\setminus\{\mathbf 0\}$ be fixed and $u_0\in V_\veck$, then problem \eqref{NSO}, with initial condition $u(0)=u_0$, has a unique global regular solution
 \beq\label{specialu}
 u(t)=e^{-tA}e^{-\Omega t S}u_0,\text{ for } t\ge 0.
 \eeq
Moreover, $u(t)$ solves the linear equation
\beq\label{linearrot}
u_t + Au + \Omega S u =0, \text{ for } t>0.
\eeq
 \end{theorem}
\begin{proof}
 Let $u_0=\sum _{m\in\Z_*} \widehat \vecu_{0,m\veck} e^{i m \widecheck\veck\cdot\vecx}\in V_\veck$.
Set $v(t)=e^{-tA}u_0$. First, we have
\beq\label{stokesv}
v_t+Av=0,\text{ for all }t>0.
\eeq

Note that $u(t)$ defined by \eqref{specialu} equals $e^{-\Omega t S} v(t)$. Then $u(0)=u_0$ and, by \eqref{stokesv},  $u(t)$ solves
\eqref{linearrot}.

Clearly, $v(t)\in V_\veck$, for all $t\ge 0$, hence,   by \eqref{specialBS},
\beq\label{strongb}
(u(t)\cdot\nabla) u(t) =0,
\eeq
which gives $B(u(t),u(t))=0$. The last fact and \eqref{linearrot} imply that $u(t)$ also solves \eqref{NSO}.

Since $u(t)$ is a global, regular solution, it is unique.
\end{proof}

\begin{remark}\label{heli}
We find some special solutions for which the helicity $\mathcal H(t)\eqdef \inprod{\nabla\times u(t),u(t)}$ vanishes for all $t\ge 0$.
(See \cite{Moffatt69,Moffatt92} for the physics of helicity, and \cite{FHN1,FHN2} for its analysis in the case of non-rotating fluids.)
We consider the solution in Theorem \ref{specthm1}, which can be written explicitly as
\beq\label{helisol}
u(t)=\sum_{m\in\Z_*} e^{-m^2|\widecheck\veck|^2 t}\Big\{\cos(\widetilde k_3\Omega t)\widehat \vecu_{0,m\veck} -\sin(\widetilde k_3\Omega t)\widetilde\veck\times \widehat \vecu_{0,m\veck}\Big\} e^{i m \widecheck\veck\cdot\vecx}.
\eeq

Then the vorticity is
\beqs
\nabla \times u(t)=\sum_{m\in\Z_*} e^{-m^2|\widecheck\veck|^2 t}im|\widecheck\veck| \Big\{\cos(\widetilde k_3\Omega t)\widetilde\veck\times\widehat \vecu_{0,m\veck} + \sin(\widetilde k_3\Omega t) \widehat \vecu_{0,m\veck}\Big\} e^{i m \widecheck\veck\cdot\vecx},
\eeqs
and the helicity is
\begin{align*}
\mathcal H(t)
&=L_1L_2L_3\sum_{m\in\Z_*} e^{-2m^2|\widecheck\veck|^2 t}im|\widecheck\veck|
\Big\{\cos^2(\widetilde k_3\Omega t)J_\veck\widehat \vecu_{0,m\veck}\cdot \overline{\widehat \vecu}_{0,m\veck}
- \sin^2(\widetilde k_3\Omega t) \widehat \vecu_{0,m\veck}\cdot J_\veck\overline{\widehat \vecu}_{0,m\veck}\\
&\quad
+ \cos(\widetilde k_3\Omega t)\sin(\widetilde k_3\Omega t)
\Big[\widehat \vecu_{0,m\veck}\cdot \overline{\widehat \vecu}_{0,m\veck}
-  J_\veck\widehat \vecu_{0,m\veck}\cdot J_\veck\overline{\widehat \vecu}_{0,m\veck}\Big]
\Big\}\\
&=L_1L_2L_3\sum_{m\in\Z_*} e^{-2m^2|\widecheck\veck|^2 t}im|\widecheck\veck|
(J_\veck\widehat \vecu_{0,m\veck}\cdot \overline{\widehat \vecu}_{0,m\veck}).
\end{align*}
Hence,
\begin{align*}
\mathcal H(t)
&= L_1L_2L_3\sum_{m\in\Z_*} e^{-2m^2|\widecheck\veck|^2 t} 2 m|\widecheck\veck|
\Big[ \big ({\rm Re}(\widehat \vecu_{0,m\veck})\times {\rm Im}(\widehat \vecu_{0,m\veck} )\big )\cdot \widetilde\veck\Big].
\end{align*}

Thus,  $\mathcal H(t)=0$ for all $t\ge 0$, provided that
\beq \label{helicond1}
\big ({\rm Re}(\widehat \vecu_{0,m\veck})\times {\rm Im}(\widehat \vecu_{0,m\veck} )\big )\cdot \widetilde\veck =0, \text{ for all }m\in \Z_*.
\eeq
However, since $\widehat \vecu_{0,m\veck}$ is orthogonal to $\widetilde\veck$ then ${\rm Re}(\widehat \vecu_{0,m\veck})\times {\rm Im}(\widehat \vecu_{0,m\veck})$ is co-linear with $\widetilde\veck$, and as result \eqref{helicond1} is equivalent to
\beq \label{helicond}
{\rm Re}(\widehat \vecu_{0,m\veck})\times {\rm Im}(\widehat \vecu_{0,m\veck})=0, \text{ for all }m\in \Z_*.
\eeq

This class of solutions \eqref{helisol}, \eqref{helicond} with vanishing helicity  for the NSE of rotating fluids has more restrictive wave vectors, i.e., the $m\veck$'s in \eqref{helisol}, than those studied in \cite[Proposition 6.4]{FHN1} for the NSE without the rotation.
\end{remark}

\begin{corollary}\label{cor53}
 There exist infinitely many vector spaces of infinite dimensions that are invariant under the NSE \eqref{NSO},
 and each space is not a subspace of any of the others.
\end{corollary}
\begin{proof}
According to Theorem \ref{specthm1}, each vector space $V_\veck$ is invariant under the NSE \eqref{NSO}, and each has infinite dimension.
 Moreover, there are infinite many $\veck$'s which are pairwise not co-linear. For those $\veck$'s, the corresponding $V_\veck$'s are the vector spaces for which the statement holds true.
\end{proof}

Next, we consider the case of non-zero spatial average solutions.

\begin{theorem}\label{specthm2}
 Suppose $u_0=\vecU_0+w_0$, where $\vecU_0$ is a constant vector in $\R^3$, and
\beqs
w_0=\sum _{m\in\Z_*} \widehat \vecu_{0,m\veck} e^{i m \widecheck\veck\cdot\vecx}\in V_\veck.
\eeqs

Let $\vecU(t)=e^{-\Omega t J}\vecU_0$ and $\vecV(t)=\int_0^t\vecU(\tau)d\tau$.
Define, for $\vecx\in\R^3$ and $t> 0$,
 \begin{align}
 \label{ulast}
\vecu(\vecx,t)&=\vecU(t)+\sum_{m\in\Z_*} e^{-m^2|\widecheck\veck|^2 t}e^{-i m \widecheck\veck\cdot \vecV(t)} E_\veck(-\widetilde k_3\Omega t)\widehat \vecu_{0,m\veck} e^{i m \widecheck\veck\cdot\vecx},\\
\label{plast}
p(\vecx,t)&=p_*(t)-\Omega \sum_{m\in\Z_*}\frac{i}{m |\widecheck\veck|} e^{-m^2|\widecheck\veck|^2 t}e^{-i m \widecheck\veck\cdot \vecV(t)} \notag \\
&\quad \cdot \big[ \cos(\widetilde k_3\Omega t) J \widetilde\veck  + \sin(\widetilde k_3\Omega t) \vece_3\big ]\cdot \widehat\vecu_{0,m\veck} e^{i m \widecheck\veck\cdot\vecx},
\end{align}
where $p_*(t)$ is any scalar function.
Then $(\vecu,p)$ is a solution of \eqref{NPDE} and \eqref{divfree} on $\R^3\times (0,\infty)$.
\end{theorem}
\begin{proof}
Let $w(t)=e^{-tA}e^{-t \Omega S} w_0$, which we write as $w(t)=\vecw(\cdot,t)$.
By Theorem \ref{specthm1}, particularly, \eqref{linearrot} and \eqref{strongb}, we have
\beqs
(\vecw\cdot\nabla) \vecw=0\text{ and } \vecw_t-\Delta \vecw + \Omega J\vecw=-\nabla q,
\eeqs
where the scalar function $q(\vecx,t)$ satisfies the geostrophic balance
\beqs
\Omega {\rm div }(J \vecw)=-\Delta q.
\eeqs

We solve this equation by
\beq\label{qone}
q(\vecx,t)=p_*(t)+\Omega (-\Delta)^{-1}{\rm div }(J \vecw(\vecx,t)),
\eeq
where the inverse operator $(-\Delta)^{-1}$ is meant to apply to functions having zero spatial average over $\domL$.

We note that $\vecw(\vecx,t)$ is Gevrey-regular, for each $t>0$. Therefore, the following calculations are valid in the classical sense.
Define
\beq \label{up}
\vecu(\vecx,t)=\vecw(\vecx-\vecV(t),t)+\vecU(t)\text{ and }p(\vecx,t)=q(\vecx-\vecV(t),t).
\eeq
(The functions $\vecu$ and $p$ in \eqref{up} will be proved to agree with those in \eqref{ulast} and \eqref{plast} later.)

Since $\vecw$ is divergence-free, then, clearly, so is $\vecu$.
We calculate, with the shorthand notation $\vecu=\vecu(\vecx,t)$ and $\vecw=\vecw(\vecx-\vecV(t),t)$,
\beqs
\vecu_t = \vecw_t - (\vecU\cdot\nabla )\vecw + \vecU'=\vecw_t - (\vecU\cdot\nabla )\vecw-\Omega J \vecU,
\eeqs
\beqs
(\vecu\cdot\nabla)\vecu=[(\vecw+\vecU)\cdot\nabla] \vecw=[\vecU\cdot\nabla] \vecw.
\eeqs
Therefore,
\begin{align*}
\Big[\vecu_t-\Delta\vecu+(\vecu\cdot\nabla)\vecu+\Omega J\vecu\Big]_{(\vecx,t)}
&=\Big[\vecw_t-\Delta \vecw  + \Omega J\vecw\Big]_{(\vecx-\vecV(t),t)} \\
&=-\Big[\nabla q\Big]_{(\vecx-\vecV(t),t)} =-\Big[\nabla p\Big]_{(\vecx,t)} .
\end{align*}

We conclude that $(\vecu,p)$ defined by \eqref{up} is a solution of \eqref{NPDE} and \eqref{divfree} on $\R^3\times (0,\infty)$.
It remains to calculate them explicitly.
First, we have
\beq\label{wlast}
\vecw(\vecx,t)=\sum_{m\in\Z_*} e^{-m^2|\widecheck\veck|^2 t}E_\veck(-\widetilde k_3\Omega t)\widehat \vecu_{0,m\veck} e^{i m \widecheck\veck\cdot\vecx}.
\eeq
Then, thanks to \eqref{up} and \eqref{wlast}, we obtain formula \eqref{ulast} for $\vecu(\vecx,t)$.
Next, we calculate ${\rm div}(J\vecw)$ by
\beq\label{divJ}
{\rm div}( J\vecw(\vecx,t))
 =\sum_{m\in\Z_*} e^{-m^2|\widecheck\veck|^2 t}i\,m \widecheck\veck^{\rm T} JE_\veck(-\widetilde k_3\Omega t)\widehat \vecu_{0,m\veck} e^{i m \widecheck\veck\cdot\vecx}.
\eeq
Denote $\vecz=\widehat \vecu_{0,m\veck} \in X_\veck$, then
\begin{align*}
\widecheck\veck^{\rm T} JE_\veck(t)\vecz
&=\cos(t)  |\widecheck\veck| \widetilde \veck \cdot  J\vecz +\sin(t) |\widecheck\veck| \widetilde \veck \cdot (\vece_3\times J_\veck \vecz)\\
&=|\widecheck\veck| \big[ \cos(t) (J^{\rm T}\widetilde \veck) \cdot \vecz -\sin(t) \vece_3\cdot J_\veck^2 \vecz\big]\\
&= |\widecheck\veck| [-\cos(t) J\widetilde\veck +\sin(t) \vece_3]\cdot \vecz.
\end{align*}
(We used the fact $J$ is anti-symmetric and relation \eqref{Jksq} below.)
Thus, together with  \eqref{qone} and \eqref{divJ},
\beq\label{qlast}
q(\vecx,t)=p_*(t)-\Omega\sum_{m\in\Z_*}\frac{i\,  e^{-m^2|\widecheck\veck|^2 t}}{m |\widecheck\veck|}
\big[ \cos(\widetilde k_3\Omega t) J \widetilde\veck  + \sin(\widetilde k_3\Omega t) \vece_3]\cdot \widehat\vecu_{0,m\veck} e^{i m \widecheck\veck\cdot\vecx}.
\eeq
From \eqref{up} and \eqref{qlast}, we obtain formula \eqref{plast} for $p(\vecx,t)$. The proof is complete.
\end{proof}

\begin{remark}\label{lim2rmk}
Let $\veck$ and $u_0$ be as in Theorem \ref{specthm1} and its proof.
We rewrite \eqref{specialu} as
\beq\label{uLam}
u(t)=\sum_{n=1}^\infty Q_{n,\Omega}(t)e^{-\Lambda_n t},
\eeq
where $Q_{n,\Omega}(t)=e^{-\Omega t S}R_{\Lambda_n}u_0$.
By \eqref{Sf1}, $Q_{n,\Omega}\in \SP_1(\mathcal V)$. Therefore, \eqref{uLam} is the asymptotic expansion of $u(t)$.

 By \eqref{etS} and \eqref{uset}, we have either $Q_{n,\Omega}=0$, or there is a unique number $m\in\N$ with $m^2 |\widecheck\veck|^2=\Lambda_n$ and
\beq\label{Qome}
Q_{n,\Omega}(t)=(\cos(\tilde k_3\Omega  t)I_3-\sin(\tilde k_3\Omega  t)J_\veck)
[ \widehat \vecu_{m\veck} e^{i m \widecheck\veck\cdot\vecx} + \widehat \vecu_{-m\veck} e^{-i m \widecheck\veck\cdot\vecx}] .
\eeq

In case $k_3=0$, we then have $Q_{n,\Omega}(t)=R_{\Lambda_n}u_0$ which is independent of $\Omega$.

Consider the case $k_3\ne 0$. Given $T>0$, define the time averaging function
 \beq\label{barQ}
 \bar Q_{n,\Omega}(t)=\frac1T\int_t^{t+T} Q_{n,\Omega}(\tau) d\tau.
 \eeq
One can see from \eqref{Qome} and \eqref{barQ} that
 \beqs
 \lim_{\Omega\to\pm\infty}\bar Q_{n,\Omega}(t)=0,\quad\text{ for any } t\in \R.
 \eeqs
\end{remark}

\appendix

\section{}\label{apd}
\begin{proof}[Proof of \eqref{etS}]
First, we have from \eqref{SF} that
\beq \label{SSt}
e^{tS}u=\sump  e^{t S_\veck}\widehat\vecu_\veck e^{i\widecheck\veck\cdot
\vecx},
\eeq
where $S_\veck=\widehat P_\veck J\widehat P_\veck $. Using formula \eqref{Pk} for $\widehat P_\veck $ and the fact $\widetilde\veck^{\rm T} J \widetilde \veck =0$,
we can compute
\beqs
S_\veck=J-\widetilde \veck \widetilde\veck^{\rm T}J-J\widetilde \veck \widetilde\veck^{\rm T}
+\widetilde \veck \widetilde\veck^{\rm T} J \widetilde \veck \widetilde\veck^{\rm T}
=J-\widetilde \veck \widetilde\veck^{\rm T}J+(\widetilde \veck \widetilde\veck^{\rm T}J)^{\rm T}.
\eeqs
With $|\widetilde \veck|^2=1$, we have
\beqs
S_\veck
=\begin{pmatrix}
  0 & -1 + \widetilde k_1^2 + \widetilde k_2^2 & \widetilde k_3 \widetilde k_2\\
  1- \widetilde k_1^2 - \widetilde k_2^2   & 0 & -\widetilde k_3 \widetilde k_1 \\
  -\widetilde k_3 \widetilde k_2 & \widetilde k_3 \widetilde k_1 &0
 \end{pmatrix}
 =\begin{pmatrix}
  0 & -\widetilde k_3^2 & \widetilde k_3 \widetilde k_2\\
  \widetilde k_3^2   & 0 & -\widetilde k_3 \widetilde k_1 \\
  -\widetilde k_3 \widetilde k_2 & \widetilde k_3 \widetilde k_1 &0
 \end{pmatrix}
 =\widetilde k_3 J_\veck.
\eeqs

Let $\vecz\in X_\veck$. Then
\beq\label{Jksq}
J_\veck^2 \vecz=\widetilde\veck\times(\widetilde\veck\times\vecz)=(\widetilde\veck\cdot \vecz)\widetilde\veck- (\widetilde\veck\cdot  \widetilde\veck) \vecz=-\vecz.
\eeq
We observe
\beqs
\ddt\big( E_\veck(t)\vecz \big)= -(\sin t)  \vecz + (\cos  t) J_\veck \vecz
=   (\sin t) J_\veck^2 \vecz + (\cos  t) J_\veck \vecz,
\eeqs
thus,
\beqs
\ddt\big( E_\veck(t)\vecz \big)
= J_\veck \big( E_\veck(t)\vecz \big).
\eeqs
This linear ODE yields the solution
$E_\veck(t)\vecz=e^{tJ_\veck} E_\veck(0)\vecz=e^{tJ_\veck}\vecz$.
Therefore,
\beq\label{Skt}
e^{tS_\veck}\vecz=e^{\widetilde k_3 tJ_\veck}\vecz=E_\veck(\widetilde k_3 t)\vecz.
\eeq
Letting $\vecz=\widehat\vecu_\veck$, we obtain \eqref{etS} thanks to \eqref{SSt} and \eqref{Skt}.
\end{proof}

\begin{proof}[Proof of Lemma \ref{intlem}]
Denote
\beqs
I(t)=\begin{pmatrix}
e^{\alpha t}\cos(\omega t)\\
e^{\alpha t}\sin(\omega t)
    \end{pmatrix}
\text{ and }
D_{-1}=\frac1{\alpha^2+\omega^2}\begin{pmatrix}
        \alpha&\omega\\
        -\omega&\alpha
       \end{pmatrix}.
\eeqs

First, we prove, for all $m\in\N\cup\{0\}$, that
\beq\label{tmI}
\int t^m I(t) dt
= \sum_{n=0}^{m}\frac{(-1)^{m-n} m!}{n!}t^n (D_{-1})^{m+1-n}I(t)+ \mathbf C,
\eeq
 where $\mathbf C$ denotes an arbitrary constant vector in $\R^2$.
Indeed, it is well-known that
\beq\label{ID}
\int I(t) dt= D_{-1} I(t) + \mathbf C,
\eeq
which proves \eqref{tmI} for $m=0$.
For $m\in\N$, integration by parts, with the use of \eqref{ID}, yields
\beqs
\int t^m I(t) dt
       = t^m D_{-1} I(t) - mD_{-1} \int  t^{m-1}  I(t) dt + \mathbf C.
\eeqs
By iterating this recursive relation, we obtain \eqref{tmI}.
Then the statement of Lemma \ref{intlem} obviously follows from \eqref{tmI}.
\end{proof}

\begin{lemma}\label{uniR}
Let $m\in \N$ and $a_n,b_n\in\C$, $\omega_n\in\R$ for $1\le n\le m$.
Suppose the function $f:\R\to\C$ defined by
\beq\label{fform}
f(t)=\sum_{n=1}^m [a_n \cos(\omega_n t)+b_n\sin(\omega_n t)], \text{ for } t\in\R,
\eeq
satisfies
\beq\label{limf}
\lim_{t\to\infty} f(t)=0.
\eeq
Then $f(t)=0$ for all $t\in \R$.
\end{lemma}
\begin{proof}
By considering the real and imaginary parts of $f$, we can assume, without loss of generality, that $a_n,b_n\in\R$ for all $n=1,\ldots,m$.
Equation \eqref{fform} can be rewritten so that $\omega_n$'s are strictly increasing non-negative numbers.
We then convert \eqref{fform}, with some re-indexing, to the following form
\beq\label{fN}
f(t)=A_0+ \sum_{n=1}^N A_n \cos(\omega_n t +\varphi_n),
\eeq
where $N\ge 0$, $A_0$, $A_n$ and $\varphi_n$, for $1\le n\le N$, are constants in $\R$, and $\omega_n$'s are positive, strictly increasing in $n$.

\emph{Claim:} $A_n=0$ for all $0\le n\le N$.

With this Claim, we have $f=0$ as desired. We now prove the Claim by induction in $N$.

Case $N=0$. Then $f(t)=A_0$, which, by \eqref{limf}, yields $A_0=0$. Therefore, the Claim is true for $N=0$.

Let $N\ge 0$. Assume the Claim is true for any function of the form \eqref{fN} that satisfies \eqref{limf}.
Now, suppose  function
\beq\label{fNN}
f(t)=A_0+ \sum_{n=1}^{N+1} A_n \cos(\omega_n t +\varphi_n)
\eeq
satisfies \eqref{limf}, with positive numbers $\omega_n$'s being strictly increasing in $n$, and $\varphi_n$'s being arbitrary numbers.

Set $T=2\pi/\omega_{N+1}>0$, and define function $g(t)=\int_t^{t+T}f(\tau)d\tau$.

On the one hand, we have, for $1\le n\le N$, that
\begin{align*}
\int_t^{t+T} \cos(\omega_n \tau+\varphi_n)d\tau
&=\frac{2}{\omega_n}\sin(\omega_n T/2) \cos(\omega_n t +\varphi_n + \omega_n T/2) \\
&=D_n \cos(\omega_n t +\varphi_n' ),
\end{align*}
where $D_n=2\omega_n^{-1} \sin(\omega_n \pi/\omega_{N+1})> 0$ and number $\varphi_n'\in\R$.
On the other hand,
\beqs
\int_t^{t+T} \cos(\omega_{N+1} \tau+\varphi_{N+1})d\tau=0.
\eeqs
Hence,
\beqs
g(t)=A_0 T+ \sum_{n=1}^N A_n D_n \cos(\omega_n t +\varphi_n').
\eeqs

Moreover, it follows from \eqref{limf} that
$g(t)\to 0$, as $t\to\infty$.
By the induction hypothesis applied to function $g$, we obtain $A_0 T=0$ and $A_n D_n=0$, for $1\le n\le N$.
Thus, $A_n=0$, for $0\le n\le N$, and \eqref{fNN} becomes
\beqs
f(t)=A_{N+1}\cos(\omega_{N+1}t+\varphi_{N+1}).
\eeqs
This form of $f$ and property \eqref{limf} imply $A_{N+1}=0$. Therefore,  the Claim holds true for $N+1$. By the induction principle, it is true for all $N\ge 0$.
\end{proof}

\section*{Acknowledgments}
The work of E.S.T.\ was supported in part by the Einstein Stiftung/Foundation - Berlin, through the Einstein Visiting Fellow Program, and by the John Simon Guggenheim Memorial Foundation. The authors would like to thank Ciprian Foias for his insights, inspiring and stimulating discussions.

\bibliography{paperbaseall}{}
\bibliographystyle{abbrv}
\end{document}